\documentclass[final,5p,times,twocolumn]{elsarticle}

\usepackage{amssymb}
\usepackage{graphicx}
\usepackage[cmex10]{amsmath}
\usepackage{array}
\usepackage{mdwtab}
\usepackage{fixltx2e}
\usepackage{graphics} 
\usepackage{epsfig} 
\usepackage{mathptmx} 
\usepackage{times} 

\usepackage{url}
\usepackage{lscape}
\usepackage{color}
\usepackage{epsfig}

\usepackage{algorithm,tabularx}
\usepackage[noend]{algpseudocode}

\makeatletter
\newcommand{\multiline}[1]{%
  \begin{tabularx}{\dimexpr\linewidth-\ALG@thistlm}[t]{@{}X@{}}
    #1
  \end{tabularx}
}
\makeatother

\usepackage[utf8x]{inputenc}
\usepackage{hyperref}
\usepackage{amssymb}
\usepackage{graphicx}
\usepackage[cmex10]{amsmath}
\usepackage{array}
\usepackage{mdwtab}
\usepackage{graphics}

\usepackage{amsthm}

\usepackage{url}
\usepackage{lscape}
\usepackage{color}
\usepackage{epsfig}
\newtheorem{theorem}{Theorem}
\newtheorem{remark}{Remark}
\newtheorem{lemma}{Lemma}

\newtheorem{assumption}{Assumption}

\newtheorem*{lemma-non}{Lemma}

\usepackage{balance}
\usepackage{dblfloatfix}
\usepackage{nicefrac}

\journal{Systems \& Control Letters}

\begin{document}

\begin{frontmatter}

\title{ROTEC: Robust to Early Termination Command Governor for Systems with Limited Computing Capacity\tnoteref{grants}}

\tnotetext[grants]{This research has been supported by National Science Foundation under award numbers ECCS-1931738, ECCS-1932530, and CMMI-1904394.}

\author[WUSTL_ESE]{Mehdi Hosseinzadeh\corref{cor1}}
\ead{mehdi.hosseinzadeh@ieee.org}

\author[WUSTL_ESE]{Bruno Sinopoli}
\ead{bsinopoli@wustl.edu}

\author[UoM]{Ilya Kolmanovsky}
\ead{ilya@umich.edu}

\author[WUSTL_CSE]{Sanjoy Baruah}
\ead{baruah@wustl.edu}

\address[WUSTL_ESE]{Department of Electrical and Systems Engineering, Washington University in St. Louis, St. Louis, MO 63130, USA}

\address[UoM]{Department of Aerospace Engineering, University of Michigan, Ann Arbor, MI 48109, USA}

\address[WUSTL_CSE]{Department of Computer Science and Engineering, Washington University in St. Louis, St. Louis, MO 63130, USA}

\cortext[cor1]{Corresponding author}

\begin{abstract}
A Command Governor (CG) is an optimization-based \textit{add-on} scheme to a nominal closed-loop system. It is used to enforce state and control constraints by modifying reference commands. This paper considers the implementation of a CG on embedded processors that have limited computing resources and must execute multiple control and diagnostics functions; consequently, the time available for CG computations is limited and may vary over time. To address this issue, a robust to early termination command governor is developed which embeds the solution of a CG problem into the internal states of a virtual continuous-time dynamical system which runs in parallel to the process. This virtual system is built so that its trajectory converges to the optimal solution (with a tunable convergence rate), and provides a sub-optimal but feasible solution whenever its evolution is terminated. This allows the designer to implement a CG strategy with a small sampling period (and consequently with a minimum degradation in its performance), while maintaining its constraint-handling capabilities. Simulations are carried out to assess the effectiveness of the developed scheme in satisfying performance requirements and real-time schedulability conditions for a practical vehicle rollover example.
\end{abstract}


\end{frontmatter}
\section{Introduction}
A common aspect of today's Cyber\textendash Physical Systems (CPSs) is that multiple safety-critical controllers/control systems responsible for different system functions may execute in a shared processing unit\textemdash see \figurename~\ref{fig:Structure}. Examples of such systems can be found in safety-critical applications like aircraft, autonomous vehicles, medical devices, and autonomous robots.

\begin{figure}[!t]
\centering
\includegraphics[width=7cm]{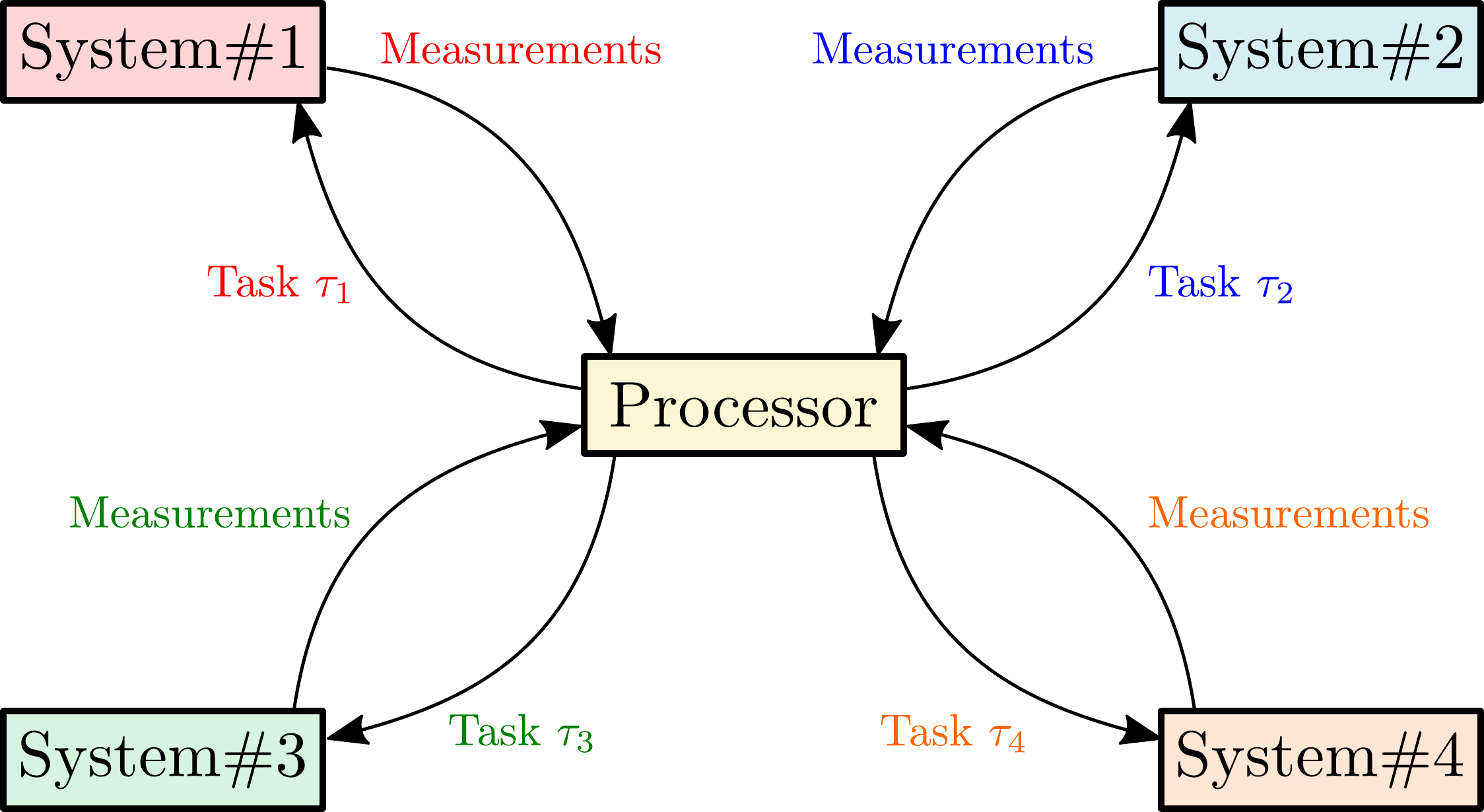}\\(a)\vspace{0.2cm}\\
\includegraphics[width=8cm]{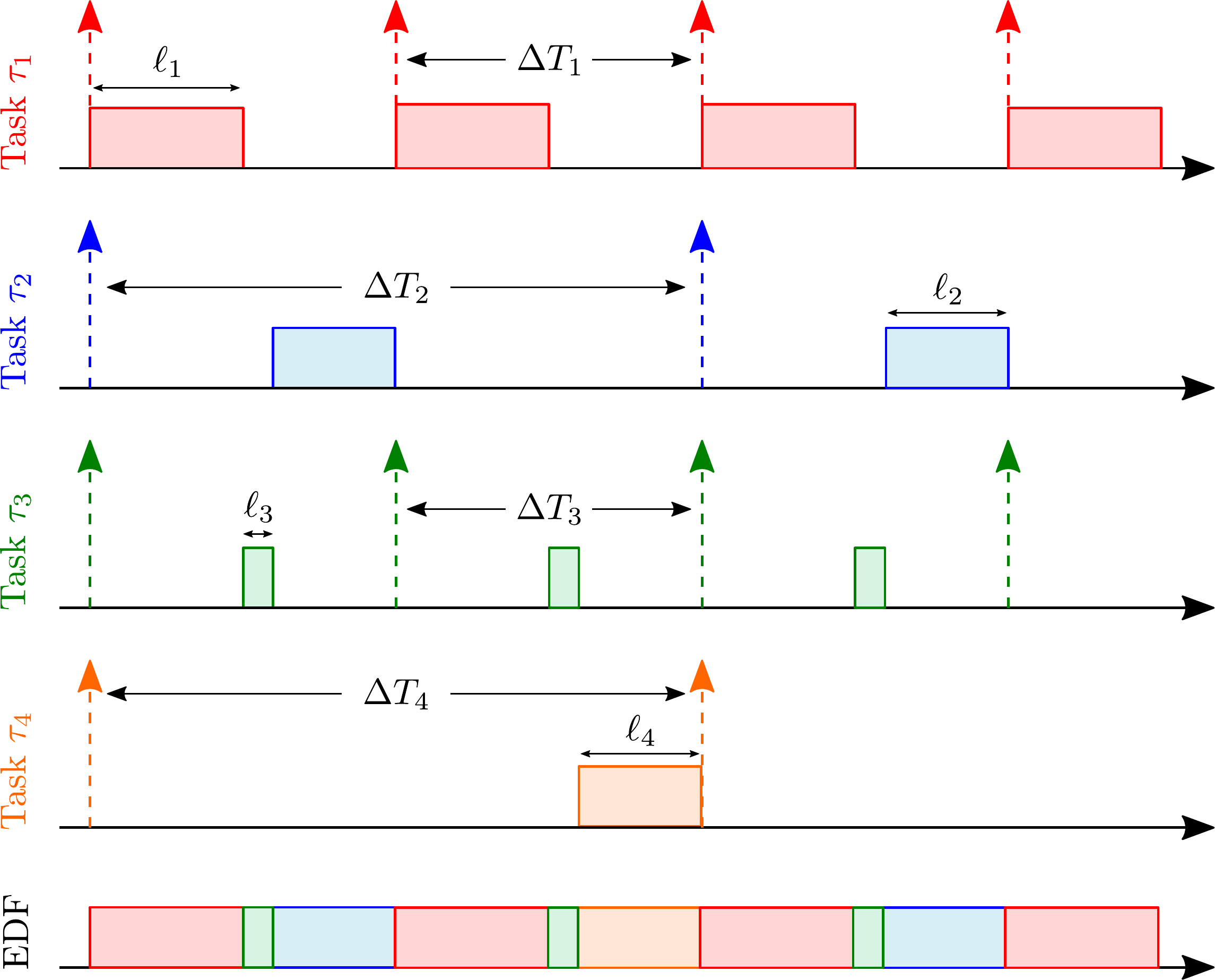}\\(b)
\caption{Figure (a): A CPS that comprises four safety-critical control systems implemented on a single processor. Task $\tau_i$ implements the controller for System\#$i,~i\in\{1,\cdots,4\}$. Figure (b): An execution schedule based on the EDF policy for the considered CPS, where $\ell_i$ and $\Delta T_i$ are the worst-case execution time and sampling period of task $\tau_i$, respectively. The upward arrows indicate when the measurements are received at the processor and tasks are invoked.}
\label{fig:Structure}
\end{figure}

Three related issues that should be addressed carefully when designing such CPSs with multiple control systems are: 1) system performance; 2) system safety; and 3) real-time schedulability. System performance refers to the degree to which the control objectives of the individual control systems are achieved, which can be indicated by various measures of cost, efficiency, accuracy, etc. System safety refers to the satisfaction of the operational constraints and requirements. Real-time schedulability assures that timing requirements of different control systems informing the CPS are satisfied.


\subsection{Prior Work on Real-Time Scheduling} \label{sec:LRScheduling}
The implementation of $N$ controllers can be seen as the problem of executing $N$ recurrent  (periodic) tasks \cite{LL73} on a resource-constrained processor. The real-time scheduling literature (e.g., \cite{Buttazzo2011}) provides a wide range of schedulability analysis techniques based upon the Earliest Deadline First (EDF) policy \cite{LL73,Der74} to address this problem. These techniques primarily emphasize the scheduling aspects and stabilization of control systems under cyber constraints (e.g., limited computational time and end-to-end delay); to the best of our knowledge, prior real-time scheduling techniques do not explicitly address/ensure physical constraints\footnote{Throughout this paper, by ``constraints" we mean constraints on system input and/or state variables.} satisfaction (e.g., hard limits on system input, state, and output variables). 

Existing real-time schedulability analysis techniques may be roughly classified into two groups. The first group (e.g., \cite{Short2011,Bini2008,Roy2021}) determines the worst-case execution time of each task, and then determines the sampling periods for each task such that the real-time schedulability conditions are satisfied.

A main shortcoming of this approach arises from the variability and unpredictability of task execution times, particularly on modern processors \cite{Wilhelm2008}. As a result, the worst-case execution times of the tasks are computed conservatively. This tends to severely under-utilize the computational resources, and requires assignment of large sampling periods to the tasks, which can lead to control performance degradation. See \cite{Lozoya2013} for more details about the trade-off between resources and control performance in embedded control systems. Though methods to determine sampling periods which ensures stability and reduce the impact on system performance have been proposed in the prior literature (e.g., \cite{Palopoli2005,Velasco2010,Chang2018,Chwa2018,Chang2017}), this paper addresses the impact of limited computing capacity on the ability to satisfy physical constraints (i.e., constraints on input and/or state variables) with the Command Governor scheme; this problem is different and has not been addressed in the prior literature.

The second group of more recent schedulability analysis techniques involves task scheduling based upon a \textit{relaxed} upper bound on their execution times. In this case, first, an optimistic upper bound on the execution time of each task is estimated, within which most invocations of the task are deemed likely to complete. Sampling periods for each task are then computed based upon these optimistic  upper bounds so to satisfy real-time schedulability conditions. At run-time, if the execution time of a task exceeds the determined upper-bound at any sampling instant, the processor allows the task to complete, which delays the execution of other tasks. This can be seen as an unwanted increase in the sampling periods of all tasks. To compensate for the impact of this increase, the control parameters for all control systems are modified at the next sampling instant. This approach has been investigated in \cite{Pazzaglia2021,Vreman2021}, and shown to lead to better performance and reduce under-utilization of  computational resources. To avoid adding extra computational burden, the modifications can be computed offline for the set of all possible cases and stored for online use. A method for reducing the number of controllers to be designed offline, while still guaranteeing specified control performance, is presented in \cite{Buttazzo2004}. However, in general, there are no systematic methods to compute the modifications, and even no guarantees for the existence of such modifications. Furthermore, proposed strategies do not enforce physical constraints.

The problem of control-scheduling co-design under different types of cyber constraints is addressed in \cite{Tabuada2007,Ryu1997,Aminifar2016}. In particular, stabilization of a control system in presence of limited execution time has been studied in  \cite{Tabuada2007}, where the authors present an event-triggered scheduler that decides which task should be executed at any given instant. The authors of \cite{Ryu1997} present a heuristic optimization method to optimize the end-to-end timing constraint (i.e., loop processing time and input-to-output latency) for a given control performance specification. A  server-based resource reservation mechanism is proposed in \cite{Aminifar2016} to ensure stability in the presence of server bandwidth constraints. However, physical constraints (i.e., constraints on input and/or state variables) have not been considered in the above-mentioned references.

Another policy used to address the real-time schedulability in CPSs is the Fixed Priority Preemptive Scheduling (FPPS) policy \cite{Kumar2020,Fara2021}. In this policy, the processor executes the highest priority task of all those tasks that are currently ready to execute. Though different methods have been proposed to assign priorities to ensure that all tasks will execute (e.g., \cite{Zhao2017,Lin2021}), if a critical task is the one with lower priority and other tasks are always schedulable without that task, the lower-priority task could wait for an unpredictable amount of time. This obviously degrades the control performance, and may even lead to constraints violation.

\subsection{Prior Work on Optimization-Based Constrained Control of Systems With Limited Computing Capacity} \label{sec:LRControl}
As mentioned above, safety in this paper is related to the constraint enforcement. The literature on constrained control has been dominated by optimization-based techniques, such as Model Predictive Control (MPC) \cite{Camacho2013,Rawlings2017}, Reference/Command Governors (RG/CG) \cite{Garone2016}, and Control Barrier Functions \cite{Ames2019}. However, the use of online optimization is computationally intensive and creates practical challenges when they are employed in a CPS with limited computational power.

Another approach to implement optimization-based control laws for system with constraints is to pre-compute them offline and store them in memory for future use. This idea is adopted in explicit MPC \cite{Alessio2009}. However, even in the case of linear quadratic MPC, explicit MPC can be more memory and computation time consuming for larger state dimensional problems or problems with a large number of constraints as compared to the onboard optimization based methods. Furthermore, it is not robust with respect to early/premature termination of the search for the polyhedral region containing the current state.

Another possible way to address limited computing power is to resort to triggering, as in self-triggered \cite{Henriksson2012} and event-triggered \cite{Yoo2021,Wang2021} optimization-based constrained control. However, there is no guarantee that  sufficient computational power will be available when the triggering mechanism invokes simultaneously multiple controllers.

Another way to address limited computational power in optimization-based constrained control schemes is to perform a fixed number of iterations to approximately track the solution of the associated optimization problem. This approach has been extensively investigated for MPC \cite{Ghaemi2009,Pherson2020}. For instance, \cite{Cimini2017} pursues the analysis of the active-set methods to determine a bound on the number of iterations. However, there is no guarantee that required iterations can be carried out in the available time in a shared processing unit. The dynamically embedded MPC is introduced in \cite{Nicotra2018}, where the processor, instead of solving the MPC iteration, runs a virtual dynamical system whose trajectory converges to the optimal solution of the MPC problem. Although employing warm-starting \cite{Wang2010} can improve convergence of the virtual dynamical system in dynamically embedded MPC, guaranteeing recursive feasibility (i.e., remaining feasible indefinitely) with this scheme is challenging, as a sudden change in the reference signal can drastically change the problem specifics, e.g., the terminal set may not be reachable anymore within the given prediction horizon. To address this issue, in \cite{Nicotra2019}, the dynamically embedded MPC is augmented with an Explicit Reference Governor (ERG) \cite{HosseinzadehMED,Cotorruelo2021,HosseinzadehECC}; however, this may lead to conservative (slow) response due to conservatism of ERG.

A CG supervisory scheme is presented in \cite{Famularo2015}, where different CG problems are designed for different operating points of the system, and a switching mechanism is proposed to switch between CG schemes. However, the effects of early termination of the computations on the CG schemes has not been considered in \cite{Famularo2015}. In \cite{Peng2019}, a one-layer recurrent neural network is proposed to solve a CG problem in finite time. However, \cite{Peng2019} does not address the variability and unpredictability of the available computing time for CG schemes when implemented in a shared processing unit and, in particular, situations when the available time is less that the required time for convergence.

\subsection{Proposed Solution} 
A CG is an optimization-based \textit{add-on} scheme to a nominal closed-loop system used to modify the reference command in order to satisfy state and input constraints. In this paper, we develop ROTEC (RObust to early TErmination Command governor), which is based on the use of primal-dual flow algorithm to define a virtual continuous-time dynamical system whose trajectory tracks the optimal solution of the CG. ROTEC runs until available time for execution runs out (that may not be known in advance) and the solution is guaranteed to enforce the constraints despite any early termination. ROTEC can guarantee safety while addressing system performance. Also, it guarantees recursive feasibility. This feature allows to satisfy the real-time schedulability condition, even when numerous safety-critical control systems are implemented on a processor with very limited computational power. 

In this paper, ROTEC is introduced as a continuous-time scheme. This facilitates the analysis and the derivation of its theoretical properties. Our numerical experiments show that these properties  are maintained when ROTEC is implemented in discrete time with a sufficiently small sampling period. This is not dissimilar to how control schemes are derived and analyzed. We leave the study of theoretical guarantees for the discrete-time implementation to future work.

\subsection{Contribution} 
To the best of our knowledge, this paper is the first that ensures robustness to early termination for CG schemes. This feature allows us to address real-time schedulability in CPSs, while ensuring constraints satisfaction at all times. In particular, it is shown analytically that ROTEC ensures the safe and efficient operation of a CPS. The main contributions are: 1) development of ROTEC; 2) demonstration that it enforces the constraints robustly with respect to the time available for computations; and 3) evaluation of its effectiveness for vehicle rollover prevention.

Our approach to CG implementation is inspired by \cite{HosseinzadehTAC,HosseinzadehLetterLuca} in exploiting barrier functions and the primal-dual continuous-time flow algorithm, but addresses a different problem. Our proofs of convergence are inspired by Lyapunov-based approaches in \cite{Nicotra2018,Nicotra2019,HosseinzadehLetter}, but once again explored for a different problem.

\subsection{Organization} 
The rest of the paper is organized as follows. Section \ref{sec:PF} formulates the problem, and discusses the control requirements and real-time schedulability conditions. Section \ref{sec:classicCG}  describes the conventional CG scheme. Section \ref{sec:ECG} develops ROTEC, proves its properties, and discusses its initialization. In Section \ref{sec:simulation}, a simulation study of vehicle rollover is reported to validate ROTEC. Finally, Section \ref{sec:conclusion} concludes the paper.

\subsection{Notation} We use $t$ to denote continuous time, $k$ to denote sampling instants, $s$ to denote predictions made at each sampling instant, and $\eta$ to denote the auxiliary time scale that ROTEC spends on solving the CG problem. $I$ and $\textbf{0}$ indicate, respectively, the identity and zero matrices with appropriate dimensions. In this paper, $\nabla_{x_1x_2}X(x_1,x_2)\triangleq\frac{\partial}{\partial x_2}(\frac{\partial}{\partial x_1}X(x_1,x_2))$.

\section{Problem Formulation}\label{sec:PF}
This section gives details about the considered CPS setting, highlights the practical challenges, and explains how we address the challenges.

\subsection{Setting} 
We consider a CPS comprising $N$ controllers implemented on a single processing unit. From a real-time computing perspective, this can be seen as a set of $N$ tasks running on a single processor. We denote the tasks by $\tau_i,~i={1,\cdots,N}$, where task $\tau_i$ performs a specific action for the $i$-th control system. The task $\tau_i$ is represented by the 2-tuple $\{\ell_i,\Delta T_i\}$, where $\ell_i$ is the worst-case execution time and $\Delta T_i$ is the sampling period. Suppose that tasks $\tau_1,\cdots,\tau_{N-1}$ correspond to pivotal actions with fixed and pre-determined sampling periods, and task $\tau_N$ implements a CG strategy. This setting is considered without loss of generality; the case in which more than one task implements a CG strategy can be addressed by applying the method to each one.

\subsection{Details of Task $\tau_N$}\label{sec:TaskN}
Suppose that task $\tau_N$ controls the following system:
\begin{align}
\dot{x}(t)=A_ox(t)+B_ou(t),
\end{align}
where $x(t)=[x_1(t)~\cdots~x_n(t)]^\top\in\mathbb{R}^n$ is the state of the system at time $t$, $u(t)=[u_1(t)~\cdots~u_p(t)]^\top\in\mathbb{R}^p$ is the control input at time $t$, and $A_o\in\mathbb{R}^{n\times n}$ and $B_o\in\mathbb{R}^{n\times p}$ are system matrices.

Although the control $u(t)$ is usually designed in the continuous-time domain, its computer implementation is described in the discrete-time domain. That is task $\tau_N$ is invoked at discrete sampling instants. Under the Logical Execution Time (LET) paradigm \cite{Henzinger2003,Frehse2014,Liang2019}, which is widely adopted in CPS, the control signal that is computed based on the measurements at sampling instant $k$ is applied to the plant at sampling instant $k+1$. This means that there is a fixed sampling-to-actuation delay which is equal to $\Delta T_N$. Thus, for a zero-order hold implementation, the sampled-data model of the plant for one-sample delay can be expressed as
\begin{align}\label{eq:systemdiscrete}
x(k+1)=A_dx(k)+B_du(k-1),
\end{align}
where $A_d=e^{A_o\Delta T_N}$ and $B_d=\int_{0}^{\Delta T_N}e^{A_ot}B_odt$ \cite{Ogata1995}.

Given the augmented state vector $z(k):=[x(k)^\top~u(k-1)^\top]^\top\in\mathbb{R}^{n+p}$ \cite{Roy2021}, system \eqref{eq:systemdiscrete} can be rewritten as:
\begin{align}\label{eq:systemfinal}
z(k+1)=Az(k)+Bu(k),
\end{align}
where 
\begin{align}
A:=\left[\begin{matrix}A_d & B_d \\ \textbf{0} & \textbf{0}\end{matrix}\right],~~B:=\left[\begin{matrix}\textbf{0} \\I_p\end{matrix}\right],
\end{align}
with $I_p$ as the $p\times p$ identity matrix.

We assume that the following nominal control law is available that stabilizes the system:
\begin{align}\label{eq:controllaw}
u(k)=Kz(k)+Gv(k),
\end{align}
where $K\in\mathbb{R}^{p\times(n+p)}$ is the feedback gain matrix, $G\in\mathbb{R}^{p\times m}$ is the feedforward gain matrix, and $v(k)\in\mathbb{R}^m$ is the command signal (a.k.a. reference commands). The feedback gain matrix $K$ should be determined such that $A_c:=A+BK$ is Schur. 

\begin{remark}
Since $rank\left(\left[\begin{matrix}\textbf{0} & B_d &  \cdots & A_d^{n-1}B_d\\ I_{p} & \textbf{0} &  \cdots & \textbf{0}\end{matrix}\right]\right)=p+rank\left(\left[\begin{matrix}B_d & \cdots & A_d^{n-1}B_d\end{matrix}\right]\right)$, it is concluded that there exists a stabilizing feedback gain matrix $K$ as in \eqref{eq:controllaw} if and only if the pair $(A_d,B_d)$ is controllable. If the pair $(A_o,B_o)$ is controllable, the sampling period can be chosen \cite{Kalman1962,Karbassi1996} such that the controllability is preserved, and consequently the existence of a stabilizing feedback gain matrix $K$ is guaranteed.
\end{remark}

\subsection{Control Requirements And Structure For Task $\tau_N$} 
Let $r(k)\in\mathbb{R}^m$ be the desired reference at sampling instant $k$. Also, let the output of System \#$N$ be defined as
\begin{align}\label{eq:output}
y(k)=Cz(k)+Dv(k),
\end{align}
where $y(k)=[y_1(k)~\cdots~y_m(k)]^\top\in\mathbb{R}^{m}$ is the output at sampling instant $k$, and $C\in\mathbb{R}^{m\times(n+p)}$ and $D\in\mathbb{R}^{m\times p}$ are output matrices. Let $\mathcal{Y}\subset\mathbb{R}^m$ be a pre-defined compact and convex set.

Suppose that task $\tau_N$ implements the CG scheme to determine $v(k)$ in \eqref{eq:controllaw} to meet the following control requirements: 
\begin{itemize}
\setlength\itemsep{0em}
\item For any desired reference $r(k)$, $y(k)\in\mathcal{Y},~\forall k$; and
\item For any constant desired reference $r$ which is inside the interior of $\mathcal{Y}$, the command signal $v(k)$ asymptotically converges to $r$, i.e., $v(k)\rightarrow r$ as $k\rightarrow\infty$.
\end{itemize}

\subsection{Real-Time Schedulability Condition} 
We consider the processor running the tasks based on EDF scheduling policy, and we assume that the deadline of each task is equal to its sampling period. Thus, the tasks $\tau_1,\cdots,\tau_N$ are schedulable if the following condition is satisfied \cite{Buttazzo2011}:
\begin{align}\label{eq:schedulabilitycondition}
U=\sum_{i=1}^N U_i\leq1,
\end{align}
where $U_i:=\ell_i/\Delta T_i$ ($U_i>0$) is called the utilization of task $\tau_i$, and $U>0$ is called the utilization of the processor.

Suppose that $\sum_{i=1}^{N-1} U_i<1$. Thus, to satisfy the real-time schedulability condition \eqref{eq:schedulabilitycondition}, the execution time and sampling period of task $\tau_N$ should satisfy $\frac{\ell_N}{\Delta T_N}\leq1-\sum_{i=1}^{N-1} U_i$. This implies that for a large $\ell_N$, the sampling period $\Delta T_N$ should be large as well. The CG problem makes use of online optimization and has a large execution time, i.e., $\ell_N$ for CG is large. Thus, $\Delta T_N$ should be set to a large value; this can degrade the performance.

\subsection{An Illustrative Example} 
Suppose that task $\tau_N$ controls the double integrator system $\dot{x}_1(t)=x_2(t)$, $\dot{x}_2(t)=u(t)$, discretized as $x_1(k+1)=x_1(k)+\Delta T_Nx_2(k)$, $x_2(k+1)=x_2(k)+\Delta T_Nu(k-1)$, controlled through $u(k)=K_1x_1(k)+K_2x_2(k)+K_3u(k-1)+Gv(k)$, with the reference signal $r=0.5$, and subject to constraints $\left\vert u(k)\right\vert\leq0.1$ and $\left\vert x_2(k)\right\vert\leq0.1$. We compute $K_1$, $K_2$, $K_3$, and $G$ such that the closed-loop poles are placed at 0.6, and for any constant command signal $v$ the equilibrium point of the system is $[v~0~0]^\top$.

Suppose that the worst-case execution time is 2 seconds, i.e., $\ell_N=2$. Assuming that $\sum_{i=1}^{N-1} U_i=0.2$, inequality \eqref{eq:schedulabilitycondition} implies that $\Delta T_N\geq2.5$ [s]. Fig.~\ref{fig:illustrativeexample} demonstrates that the tracking performance is degraded with larger $\Delta T_N$.

\subsection{Goal of This Paper}  
The main goal of this paper to develop a method, called ROTEC, to implement the CG scheme without requiring the exact optimization. The core idea is to use the primal-dual gradient flow to track the optimal solution of the CG problem, and provide a feasible solution if terminated at any moment.

\section{Conventional Command Governor}\label{sec:classicCG}
The core idea behind the CG scheme is to augment a prestabilized system with an \textit{add-on} control unit that, whenever necessary, manipulates the command signal to ensure constraint satisfaction. In the following, we briefly describe the basics of CG under the following assumptions.

\begin{assumption}\label{assum:observability}
$A+BK$ is Schur and $(A,C)$ is observable, where $A,B$ are as in \eqref{eq:systemfinal}, $K$ is as in \eqref{eq:controllaw}, and $C$ is as in \eqref{eq:output}.
\end{assumption}

\begin{assumption}\label{assum:constraintset}
The set $\mathcal{Y}$ is defined as $\mathcal{Y}:=\{y|y_i\leq\bar{y}_i,~\forall i\}$, where $\bar{y}_i\geq0$. This assumption is not restrictive, as any convex domain with nonempty interior can be inner-approximated with a polyhedron \cite{Bemporad2004}. 
\end{assumption}

Suppose that task $\tau_N$ implements CG to control system \eqref{eq:systemfinal} through the control law \eqref{eq:controllaw}. At any sampling instant $k$, CG computes the optimal command signal $v^\ast(k)$ by solving the following optimization problem \cite{Garone2016}:
\begin{align}\label{eq:CG1}
v^\ast(k)=\left\{
\begin{array}{ll}
     &  \arg\min\limits_{v}\frac{1}{2}\left\Vert v-r(k)\right\Vert_Q^2\\
    \text{s.t.} & (z(k),v)\in \tilde{O}_\infty 
\end{array}
\right.,
\end{align}
where $Q=Q^\top>0$, $\left\Vert v-r(k)\right\Vert_Q^2=(v-r(k))^\top Q(v-r(k))$, and $\tilde{O}_\infty$ is a subset of the maximal output admissible set:
\begin{align}
O_\infty=\{(z,v)|\hat{y}_i(s|z,v)\leq\bar{y}_i,~i=1,\cdots,m,~s=0,1,\cdots\},
\end{align}
where $\hat{y}(s|z,v)=[\hat{y}_1(s|z,v)~\cdots~\hat{y}_m(s|z,v)]^\top$ is the predicted output at the prediction instant $s$, which according to \eqref{eq:systemfinal} and \eqref{eq:controllaw}-\eqref{eq:output} can be computed as
\begin{align}\label{eq:prediction}
\hat{y}_i(s|z,v)=&C_iA_c^sz+H_{is}v,
\end{align}
where $H_{is}:=C_i(I-A_c)^{-1}(I-A_c^s)BG+D_i$ is a constant nonzero vector, with $C_i$ and $D_i$ as the $i$-th row of output matrices $C$ and $D$, respectively. Note that CG computes the optimal command signal such that the predicted response from the initial condition $z(k)$ with the command signal kept constant satisfies the constraints. As a result, the predictions \eqref{eq:prediction} are computed by fixing the command signal throughout all $s$ steps.

\begin{figure}[!t]
\centering
\includegraphics[width=5.85cm]{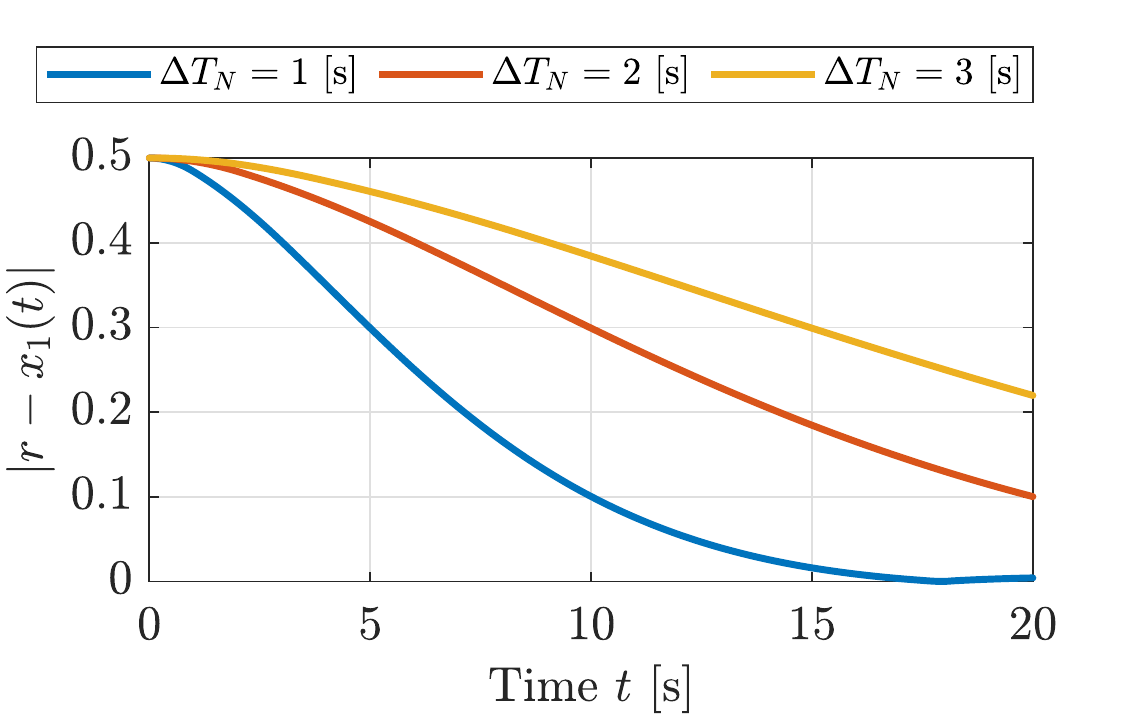}
\caption{The tracking performance of CG for different sampling period $\Delta T_N$.}
\label{fig:illustrativeexample}
\end{figure}

Since $A_c$ is Schur, Assumption \ref{assum:observability} and \ref{assum:constraintset} imply that \cite{Gilbert1991} the set $\tilde{O}_\infty=O_\infty\bigcap\Gamma$,
where $\Gamma=\{(z,v)|\hat{y}_i(\infty|z,v)\leq (1-\epsilon)\bar{y}_i,~i=1,\cdots,m\}$ for some $\epsilon>0$, is finitely determined and positively invariant. That is there exists a finite index $s^\ast$ such that
\begin{align}\label{eq:Oinfty}
\tilde{O}_\infty=\{(z,v)|\hat{y}_i(s|z,v)\leq\bar{y}_i,~i=1,\cdots,m,~s=0,1,\cdots,s^\ast\}\bigcap\Gamma. 
\end{align}

The value of index $s^\ast$ can be obtained by solving a sequence of mathematical programming  problems which is detailed in \cite[Algorithm 3.2]{Gilbert1991}. The computations are often straightforward, even when $p$, $m$ and $n$ are quite large. In particular, when $\mathcal{Y}$ is polyhedral (as in our case; see Assumption 2), the programming problems are linear. Note that these computations can be performed once and are offline.

Therefore, the CG problem given in \eqref{eq:CG1} can be rewritten as
\begin{align}\label{eq:CG2}
v^\ast(k)=\left\{
\begin{array}{ll}
     &  \arg\min\limits_{v}\frac{1}{2}\left\Vert v-r(k)\right\Vert_Q^2\\
    \text{s.t.} & \hat{y}_i(s|z(k),v)\leq\bar{y}_i,~i=1,\cdots,m,s=0,1,\cdots,s^\ast\\
    & \hat{y}_i(\infty|z(k),v)\leq (1-\epsilon)\bar{y}_i,~i=1,\cdots,m
\end{array}
\right.
\end{align}
which is a Quadratic Programming (QP) problem with $m\cdot(s^\ast+2)$ linear inequality constraints.

\begin{remark}\label{remark:Slater}
Since $\tilde{O}_\infty$ is positively invariant, if $(z(0),v(0))\in\tilde{O}_\infty$, at any sampling instant $k$, there exists $v^\ast(k)$ such that $\hat{y}_i(s|z(k),v^\ast(k))\leq\bar{y}_i,~ i\in\{1,\cdots,m\},~s\in\{1,\cdots,s^\ast\}$ and $\hat{y}_i(\infty|z(k),v^\ast(k))\leq(1-\epsilon)\bar{y}_i,~i\in\{1,\cdots,m\}$.
\end{remark}

\begin{remark}\label{remark:Geometric1}
The Karush–Kuhn–Tucker (KKT) condition \cite{Boyd2004} implies that, at any sampling instant $k$, $\hat{y}_{i^\dag}(s^\dag|z(k),v^\ast(k))=\bar{y}_{i^\dag}$ if the constraint is active, and $\hat{y}_{i^\dag}(s^\dag|z(k),v^\ast(k))<\bar{y}_{i^\dag}$ otherwise, where $i^\dag\in\{1,\cdots,m\}$ and $s^\dag\in\{0,\cdots,s^\ast\}$. Also, $\hat{y}_{i^\dag}(\infty|z(k),v^\ast(k))=(1-\epsilon)\bar{y}_{i^\dag}$ if the constraint is active, and $\hat{y}_{i^\dag}(\infty|z(k),v^\ast(k))<(1-\epsilon)\bar{y}_{i^\dag}$ otherwise.
\end{remark}

\section{Proposed Solution: ROTEC}\label{sec:ECG}
A common approach to solve the optimization problem \eqref{eq:CG2} is to use the primal-dual interior-point methods. Though these methods are fast and efficient, in general, the iterates in these methods are not necessarily feasible \cite[pp. 609]{Boyd2004}. Thus, in the presence of early termination, to ensure constraint satisfaction one could resort to keeping the command signal unchanged (note that $v^\ast(k)$ is feasible at sampling instant $k+1$; see Remark \ref{remark:Slater}), but this could degrade system performance. Another approach to solve \eqref{eq:CG2} is to use the primal barrier interior-point methods. The main weakness of these methods is that they require a high number of Newton steps when high accuracy is required \cite[pp. 569]{Boyd2004}. Also, in general, there is no guarantee that early termination yields a feasible point \cite{Wang2010}.

In this section, we develop ROTEC to address the practical challenges discussed above. We begin by tightening the constraints of the conventional CG given in \eqref{eq:CG2}. Then, we build a continuous-time dynamical system that tracks the optimal solution of the CG problem, characterize its convergence properties, and define the ROTEC algorithm.

\subsection{Constraint Tightening}\label{sec:ConstraintTightening}
Consider the following optimization problem:
\begin{align}\label{eq:CG3}
v^\dag(k)=\left\{
\begin{array}{ll}
     &  \arg\min\limits_{v}\frac{1}{2}\left\Vert v-r(k)\right\Vert_Q^2\\
    \text{s.t.} & f_{is}(z(k),v)\leq0,~i=1,\cdots,m,~s=0,\cdots,s^\ast,\infty
\end{array}
\right.
\end{align}
where
\begin{align}\label{eq:f_is}
\left\{
\begin{array}{l}
f_{is}(z(k),v):=\hat{y}_i(s|z(k),v)-\bar{y}_i+1/\beta,~s\in\{0,\cdots,s^\ast\}\\
f_{i\infty}(z(k),v):=\hat{y}_i(\infty|z(k),v)-(1-\epsilon)\bar{y}_i+1/\beta
\end{array}
\right.,
\end{align}
with sufficiently large $\beta>0$ to make sure that the feasible set of \eqref{eq:CG3} is nonempty. See \figurename~\ref{fig:GeometricIllustration2} for a geometric illustration of $v^\ast(k)$ and $v^\dag(k)$. Note that the larger the $\beta$, the closer the optimization problem \eqref{eq:CG3} to \eqref{eq:CG2}, that is $\lim_{\beta\rightarrow\infty}v^\dag(k)=v^\ast(k)$.

The main advantage of the constraint tightening in \eqref{eq:f_is} is that it allows us define a continuous-time dynamical system whose trajectory remains feasible at all times and tracks the optimal solution (as will be proven in Theorems \ref{theorem:convergence} and \ref{theorem:convergence1}).

\subsection{Continuous-Time Dynamical System}
The \textit{modified} barrier function \cite{Polyak1992} associated with the optimization problem \eqref{eq:CG3} is\footnote{Note that despite \cite{Polyak1992}, we do not consider the multiplier $\frac{1}{\beta}$ in the penalty terms.}
\begin{align}\label{eq:BarrierFunction}
\mathcal{B}(z(k),v,\lambda)=&\frac{1}{2}\left\Vert v-r(k)\right\Vert_Q^2-\sum\limits_{i=1}^m\lambda_{i\infty}\log(-\beta f_{i\infty}(z(k),v)+1)\nonumber\\
&-\sum\limits_{i=1}^m\sum\limits_{s=0}^{s^\ast}\lambda_{is}\log(-\beta f_{is}(z(k),v)+1),
\end{align}
which can be seen as the Lagrangian for the following problem\footnote{In the rest of the paper, $\forall i$ means given any element of the set $\{1,\cdots,m\}$ and $\forall s$ means given any element of the set $\{0,\cdots,s^\ast,\infty\}$.}
\begin{align}\label{eq:CG4}
v^\dag(k)=\left\{
\begin{array}{ll}
     &  \arg\min\limits_{v}\frac{1}{2}\left\Vert v-r(k)\right\Vert_Q^2\\
    \text{s.t.} & \log\big(-\beta f_{is}(z(k),v)+1\big)\geq0,~\forall i,s
\end{array}
\right.,
\end{align}
with $\lambda=[\lambda_{10}~\cdots~\lambda_{m\infty}]^\top\in\mathbb{R}^{m(s^\ast+2)}$ ($\lambda_{is}\geq0,~\forall i,s$) as the Lagrange multiplier. At any sampling instant $k$, we denote the optimal dual parameter by $\lambda^\dag(k)$. 

\begin{figure}[!t]
\centering
\includegraphics[width=\columnwidth]{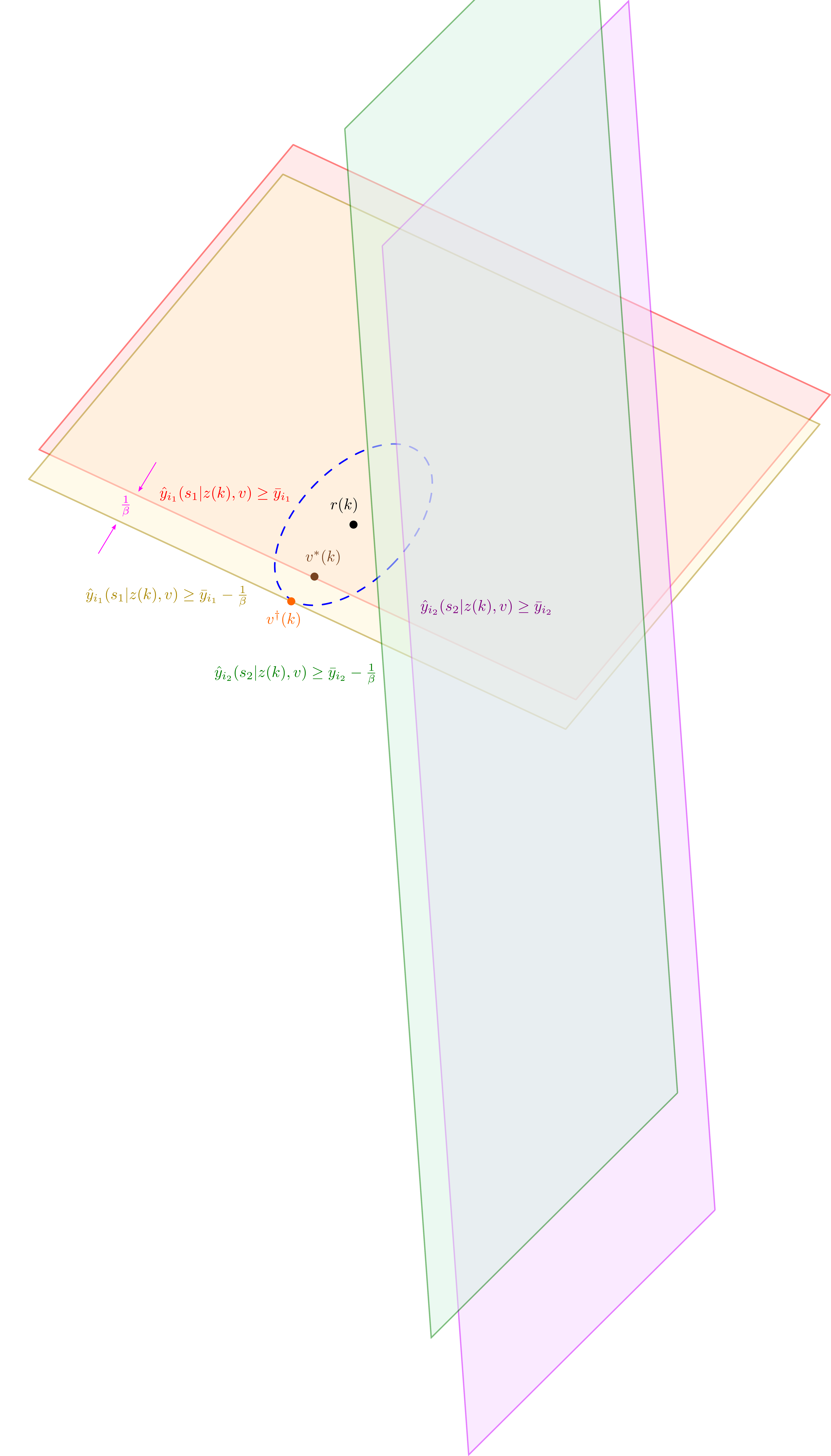}
\caption{A geometric illustration of the tightened constraints and the optimal solution $v^\dag(k)$, where $i_1,i_2\in\{1,\cdots,m\}$ and $s_1,s_2\in\{0,\cdots,s^\ast\}$.}
\label{fig:GeometricIllustration2}
\end{figure}

\begin{figure*}[!t]
\setcounter{equation}{16}
\begin{align}
\frac{d}{d\eta}\hat{v}(\eta)=&-\sigma\nabla_{\hat{v}}\mathcal{B}(k,\eta)=-\sigma\left(Q(\hat{v}(\eta)-r(k))+\beta\sum\limits_{i=1}^m\sum\limits_{s=0}^{s^\ast}\hat{\lambda}_{is}(\eta)\frac{\left(H_{is}\right)^\top}{-\beta f_{is}(k,\eta)+1}+\beta\sum\limits_{i=1}^m\hat{\lambda}_{i\infty}(\eta)\frac{\left(H_{is}\right)^\top}{-\beta f_{i\infty}(k,\eta)+1}\right),\label{eq:SystemUpdate1}\\
\frac{d}{d\eta}\hat{\lambda}_{is}(\eta)=&\sigma\big(\nabla_{\hat{\lambda}_{is}}\mathcal{B}(k,\eta)+\Psi_{is}(k,\eta)\big)=\sigma\Big(-\log\big(-\beta f_{is}(k,\eta)+1\big)+\Psi_{is}(k,\eta)\Big);\label{eq:SystemUpdate2}\\
\Psi_{is}(k,\eta)=&\left\{
\begin{array}{rl}
   0,  & \text{if (}\hat{\lambda}_{is}(\eta)>0\text{)} \text{ OR } \text{(}\hat{\lambda}_{is}(\eta)=0 \text{ and } \log\big(-\beta f_{is}(k,\eta)+1\big)<0 \text{)}\\
   \log\big(-\beta f_{is}(k,\eta)+1\big),  & \text{if }\hat{\lambda}_{is}(\eta)=0 \text{ and } \log\big(-\beta f_{is}(k,\eta)+1\big)>0
\end{array}
\right..\label{eq:Phi}
\end{align}
\hrule
\end{figure*}

\begin{remark}\label{remark:KKT}
According to the KKT condition, for active constraints we have $\log(-\beta f_{is}(z(k),v^\dag(k))+1)=0$ and $\lambda_{is}^\dag(k)\geq0$, and for inactive constraints we have $\log(-\beta f_{is}(z(k),v^\dag(k))+1)>0$ and $\lambda_{is}^\dag(k)=0$. We have $\log(-\beta f_{is}(z(k),v^\dag(k))+1)=0\Leftrightarrow f_{is}(z(k),v^\dag(k))=0$ and $\log(-\beta f_{is}(z(k),v^\dag(k))+1)>0\Leftrightarrow f_{is}(z(k),v^\dag(k))<0$, which means that active/inactive constraints of problems \eqref{eq:CG3} and \eqref{eq:CG4} are the same. 
\end{remark}

\begin{remark}
Since $\log(-\beta f_{is}(z(k),v)+1)\geq0$ if and only if $f_{is}(z(k),v)\leq0$, it implies \cite[pp. 131]{Boyd2004} that the optimal solutions of problems \eqref{eq:CG3} and \eqref{eq:CG4} are identical.  
\end{remark}

\begin{remark}\label{remark:saddle}
Let $\tilde{O}_{\infty,1/\beta}$ be as in \eqref{eq:Oinfty}, where constraints are tightened by $1/\beta$. Note that $\tilde{O}_{\infty,1/\beta}$ is positively invariant. Thus, if $(z(0),v(0))\in\tilde{O}_{\infty,1/\beta}$, at any sampling instant $k$, there exists $v^\dag(k)$ such that $\log(-\beta f_{is}(z(k),v^\dag(k))+1)\geq0,~\forall i,s$.
\end{remark}

At this stage, we propose the primal-dual gradient flow as shown in \eqref{eq:SystemUpdate1}-\eqref{eq:SystemUpdate2} which should be implemented at any sampling instant $k$, where $\sigma>0$ is a design parameter and $\eta$ is the auxiliary time variable\footnote{For the sake of brevity, we will denote $\mathcal{B}(z(k),\hat{v}(\eta),\hat{\lambda}(\eta))$ and $f_{is}(z(k),\hat{v}(\eta))$ by $\mathcal{B}(k,\eta)$ and $f_{is}(k,\eta)$, respectively, when the main focus is the time stamps $k$ and $\eta$.}. The function $\Psi_{is}(k,\eta)$ given in \eqref{eq:Phi} is the projection operator onto the normal cone of $\lambda$ \cite{Ryu2016}. The differential equations \eqref{eq:SystemUpdate1} and \eqref{eq:SystemUpdate2} build a virtual continuous-time system whose properties will be discussed next.

\begin{remark}
Given $\theta=[\hat{v}~\hat{\lambda}]$, it can be shown that the Hessian matrix $\nabla_{\theta\theta}\mathcal{B}(k,\eta)$ is not full rank. Thus, we cannot use methods requiring the inverse of $\nabla_{\theta\theta}\mathcal{B}(k,\eta)$, e.g., \cite{Fazlyab2016}. Other methods that use an approximation of $\left(\nabla_{\theta\theta}\mathcal{B}(k,\eta)\right)^{-1}$, like quasi-Newton
method, have a slower convergence \cite{Gilli2019}. 
\end{remark}

\subsection{Properties}
In this subsection, we prove convergence (Theorem \ref{theorem:convergence}) and constraint-handling (Theorem \ref{theorem:convergence1}) properties of system \eqref{eq:SystemUpdate1}-\eqref{eq:SystemUpdate2}. First, we show that $\left[\big(\nabla_{\hat{v}}\mathcal{B}(k,\eta)\big)^\top~-\big(\nabla_{\hat{\lambda}}\mathcal{B}(k,\eta)\big)^\top\right]^\top$ is strongly monotone, which will be used in the proof of Theorem \ref{theorem:convergence}.

\begin{lemma}\label{lemma:monotonicity}
The operator $\left[\big(\nabla_{\hat{v}}\mathcal{B}(k,\eta)\big)^\top~-\big(\nabla_{\hat{\lambda}}\mathcal{B}(k,\eta)\big)^\top\right]^\top$ is strongly monotone w.r.t. $(\hat{v},\hat{\lambda})$. That is $\exists\mu>0$ such that
\setcounter{equation}{19}
\begin{align}
\left[\begin{matrix}\nabla_{\hat{v}}\mathcal{B}(k,\eta)\\-\nabla_{\hat{\lambda}}\mathcal{B}(k,\eta)\end{matrix}\right]^\top\left[\begin{matrix}\hat{v}(\eta)-v^\dag(k)\\ \hat{\lambda}(\eta)-\lambda^\dag(k)\end{matrix}\right]\geq\mu\left\Vert\left[\begin{matrix}\hat{v}(\eta)-v^\dag(k)\\ \hat{\lambda}(\eta)-\lambda^\dag(k)\end{matrix}\right]\right\Vert^2.
\end{align}

\end{lemma}

\begin{proof}
The Jacobian of the operator is
\begin{align}
\textbf{J}=\left[\begin{matrix}\nabla_{\hat{v}\hat{v}}\mathcal{B}(k,\eta) & \nabla_{\hat{v}\hat{\lambda}}\mathcal{B}(k,\eta)\\ -\nabla_{\hat{\lambda}\hat{v}}\mathcal{B}(k,\eta) & -\nabla_{\hat{\lambda}\hat{\lambda}}\mathcal{B}(k,\eta)
\end{matrix}\right],
\end{align}
where $\nabla_{\hat{\lambda}\hat{\lambda}}\mathcal{B}(k,\eta)=\textbf{0}$, $\nabla_{\hat{v}\hat{\lambda}}\mathcal{B}(k,\eta)=\nabla_{\hat{\lambda}\hat{v}}\mathcal{B}(k,\eta)$, and $\nabla_{\hat{v}\hat{v}}\mathcal{B}(k,\eta)$ is 
\begin{align}\label{eq:Hessian}
\nabla_{\hat{v}\hat{v}}\mathcal{B}(k,\eta)=&Q+\beta^2\sum\limits_{i=1}^m\sum\limits_{s=0}^{s^\ast}\hat{\lambda}_{is}(\eta)\frac{ \left(H_{is}\right)^\top H_{is}}{(-\beta f_{is}(k,\eta)+1)^2}\nonumber\\
&+\beta^2\sum\limits_{i=1}^m\hat{\lambda}_{i\infty}(\eta)\frac{ \left(H_{is}\right)^\top H_{is}}{(-\beta f_{i\infty}(k,\eta)+1)^2},
\end{align}
which is positive definite as $\hat{\lambda}_{is}(\eta)\geq0,~\forall i,s,\eta$ and $Q>0$. Thus, $\textbf{J}+\textbf{J}^\top>0$, which implies \cite{Ryu2016} that the operator is strongly monotone. This completes the proof. 
\end{proof}

The following theorem shows that the trajectory of system \eqref{eq:SystemUpdate1}-\eqref{eq:SystemUpdate2} converges to the optimal solution $(v^\dag(k),\lambda^\dag(k))$.

\begin{theorem}\label{theorem:convergence}
Let $v^\dag(k)$ be as in \eqref{eq:CG4} and $(\hat{v}(0),\hat{\lambda}(0))$ be the feasible initial condition
for system \eqref{eq:SystemUpdate1}-\eqref{eq:SystemUpdate2}. Then, $(\hat{v}(\eta),\hat{\lambda}(\eta))$ exponentially converges to $(v^\dag(k),\lambda^\dag(k))$ as $\eta\rightarrow\infty$. 
\end{theorem}

\begin{proof}
Consider the following Lyapunov function:
\setcounter{equation}{22}
\begin{align}\label{eq:Lyapunov}
V(\hat{v}(\eta),\hat{\lambda}(\eta))=&\frac{1}{2\sigma}\left\Vert \hat{v}(\eta)-v^\dag(k)\right\Vert^2+\frac{1}{2\sigma}\left\Vert \hat{\lambda}(\eta)-\lambda^\dag(k)\right\Vert^2,
\end{align}
whose time derivative w.r.t. the auxiliary time variable $\eta$ is
\begin{align}\label{eq:Vdot1}
\frac{d}{d\eta}V(\hat{v}(\eta),\hat{\lambda}(\eta))=&\frac{1}{\sigma}\left(\hat{v}(\eta)-v^\dag(k)\right)^\top\frac{d}{d\eta}\hat{v}(\eta)\nonumber\\
&+\frac{1}{\sigma}\left(\hat{\lambda}(\eta)-\lambda^\dag(k)\right)^\top\frac{d}{d\eta}\hat{\lambda}(\eta),
\end{align}
where $\hat{\lambda}(\eta)=[\hat{\lambda}_{10}(\eta)~\cdots~\hat{\lambda}_{m\infty}(\eta)]^\top\in\mathbb{R}^{m(s^\ast+2)}$. According to \eqref{eq:SystemUpdate1} and \eqref{eq:SystemUpdate2}, it follows from \eqref{eq:Vdot1} that
\begin{align}\label{eq:Vdot2}
\frac{d}{d\eta}V(\cdot)=&-\left(\hat{v}(\eta)-v^\dag(k)\right)^\top\nabla_{\hat{v}}\mathcal{B}(\cdot)\nonumber\\
&+\big(\hat{\lambda}(\eta)-\lambda^\dag(k)\big)^\top\big(\nabla_{\hat{\lambda}}\mathcal{B}(\cdot)+\Psi(\cdot)\big),
\end{align}
where $\Psi(k,\eta)=[\Psi_{10}(k,\eta)~\cdots~\Psi_{m\infty}(k,\eta)]^\top\in\mathbb{R}^{m(s^\ast+2)}$.

According to \eqref{eq:Phi}, $\Psi_{is}(\cdot)$ is $-\nabla_{\hat{\lambda}_{is}}\mathcal{B}(k,\eta)$ when $\hat{\lambda}_{is}(\eta)=0$ and $\nabla_{\hat{\lambda}_{is}}\mathcal{B}(k,\eta)<0$, and is zero otherwise. Thus, since $\lambda_{is}^\dag(k)\geq0$ for any $k$, it implies that $(\hat{\lambda}_{is}(\eta)-\lambda_{is}^\dag(k))\cdot(\nabla_{\hat{\lambda}_{is}}\mathcal{B}(k,\eta)+\Psi_{is}(k,\eta))\leq(\hat{\lambda}_{is}(\eta)-\lambda_{is}^\dag(k))\nabla_{\hat{\lambda}_{is}}\mathcal{B}(k,\eta)$. Thus, it follows from \eqref{eq:Vdot2}:
\begin{align}\label{eq:Vdot3}
\frac{d}{d\eta}V(\hat{v}(\eta),\hat{\lambda}(\eta))\leq&-\left(\hat{v}(\eta)-v^\dag(k)\right)^\top\nabla_{\hat{v}}\mathcal{B}(k,\eta)\nonumber\\
&+(\hat{\lambda}(\eta)-\lambda^\dag(k))^\top\nabla_{\hat{\lambda}}\mathcal{B}(k,\eta),
\end{align}
which together with Lemma \ref{lemma:monotonicity} implies that
\begin{align}\label{eq:Vdot4}
\frac{d}{d\eta}V(\cdot)\leq&-\mu\left\Vert\left[\begin{matrix}\hat{v}(\eta)-v^\dag(k)\\ \hat{\lambda}(\eta)-\lambda^\dag(k)\end{matrix}\right]\right\Vert^2=-2\sigma\mu V(\cdot). 
\end{align}

Therefore, 
\begin{align}
V\big(\hat{v}(\eta),\hat{\lambda}(\eta)\big)\leq V\big(\hat{v}(0),\hat{\lambda}(0)\big)\cdot e^{-2\sigma\mu\eta}, 
\end{align}
and consequently,
\begin{align}\label{eq:ExponentialConvergence}
\left\Vert\left[\begin{matrix}\hat{v}(\eta)-v^\dag(k)\\ \hat{\lambda}(\eta)-\lambda^\dag(k)\end{matrix}\right]\right\Vert^2\leq \left\Vert\left[\begin{matrix}\hat{v}(0)-v^\dag(k)\\ \hat{\lambda}(0)-\lambda^\dag(k)\end{matrix}\right]\right\Vert^2e^{-2\sigma\mu\eta},
\end{align}
which completes the proof. 
\end{proof}

Theorem \ref{theorem:convergence} showed that the trajectory of system \eqref{eq:SystemUpdate1}-\eqref{eq:SystemUpdate2} converges to the optimal solution $(v^\dag(k),\lambda^\dag(k))$.  However, the evolution of system \eqref{eq:SystemUpdate1}-\eqref{eq:SystemUpdate2} might be terminated before convergence due to limited computation power. Thus, the trajectory of system \eqref{eq:SystemUpdate1}-\eqref{eq:SystemUpdate2} must remain feasible at all times. Theorem \ref{theorem:convergence1} formally ensures this property for the virtual continuous-time system \eqref{eq:SystemUpdate1}-\eqref{eq:SystemUpdate2}. Before that, first, we pose Remark \ref{remark:dualparameter} which will be used in the proof of Theorem \ref{theorem:convergence1}.

\begin{remark}\label{remark:dualparameter}
According to \eqref{eq:SystemUpdate2}-\eqref{eq:Phi}, $\frac{d}{d\eta}\hat{\lambda}_{is}(\eta)>0$ if $0<f_{is}(k,\eta)<1/\beta$. Thus, when $f_{is}(k,\eta)$ is in close proximity of $1/\beta$: i) there exists $\underline{\lambda}>0$ such that $\hat{\lambda}_{is}(\eta)\geq\underline{\lambda}$; and ii) $\Psi_{is}(k,\eta)=0$.
\end{remark}

\begin{figure*}[!b]
\hrule
\setcounter{equation}{33}
\begin{align}
\left\Vert\nabla_{\hat{v}}\mathcal{B}(\cdot)\right\Vert^2=&\left(\sum\limits_{j=1}^\xi\frac{\beta\hat{\lambda}_{i_js_j}(\cdot)}{\phi_{i_js_j}(\cdot)}\left(H_{i_js_j}\right)^\top\right)^\top\left(\sum\limits_{j=1}^\xi\frac{\beta\hat{\lambda}_{i_js_j}(\cdot)}{\phi_{i_js_j}(\cdot)}\left(H_{i_js_j}\right)^\top\right)+O(1)=\frac{\beta^2}{\left(\underline{\phi}(\cdot)\right)^2}\left\Vert\sum\limits_{j=1}^\xi\frac{\underline{\phi}(\cdot)}{\phi_{i_js_j}(\cdot)}\hat{\lambda}_{i_js_j}(\cdot)\left(H_{i_js_j}\right)^\top\right\Vert^2+O(1).
\label{eq:limitbehavior3}
\end{align}
\setcounter{equation}{35}
\begin{align}\label{eq:limitB2}
\lim\limits_{\phi_{i_js_j}(k,\eta)\rightarrow0^+,~\forall j}\;\left(\frac{d}{d\eta}\mathcal{B}(k,\eta)\right)\leq&\lim\limits_{\underline{\phi}(k,\eta)\rightarrow0^+}\;-\sigma\left(\frac{\zeta\beta^2\underline{\lambda}^2}{\left(\underline{\phi}(k,\eta)\right)^2}-\xi\left(\log(\underline{\phi}(k,\eta))\right)^2\right)<0.
\end{align}
\end{figure*}

\begin{theorem}\label{theorem:convergence1}
Let $(\hat{v}(\eta),\hat{\lambda}(\eta))$ be the solution of \eqref{eq:SystemUpdate1}-\eqref{eq:SystemUpdate2}. Given a feasible $(\hat{v}(0),\hat{\lambda}(0))$, $\hat{v}(\eta)$ satisfies the constraints of the conventional CG problem at all $\eta$. 
\end{theorem}

\begin{proof}
Let $\phi_{is}(k,\eta):=-\beta f_{is}(k,\eta)+1$. According to \eqref{eq:f_is}, the constraints of the conventional CG problem \eqref{eq:CG2} are satisfied if $\phi_{is}(k,\eta)>0,~\forall i,s$. Note that $\mathcal{B}(k,\eta)\rightarrow\infty$ only if $\phi_{i_js_j}(k,\eta)\rightarrow0^+,~j\in\{1,\cdots,\xi\}$, where $\xi\leq m$, and $i_j\in\{1,\cdots,m\}$ and $s_j\in\{0,\cdots,s^\ast,\infty\}$. Thus, the boundedness of $\mathcal{B}(k,\eta)$ from above is equivalent to the constraint satisfaction at all $\eta$. 

We prove the boundedness of $\mathcal{B}(k,\eta)$ by showing that
\setcounter{equation}{29}
\begin{align}\label{eq:limitB1}
\lim\limits_{\phi_{i_js_j}(k,\eta)\rightarrow0^+,~\forall j}\;\left(\frac{d}{d\eta}\mathcal{B}(k,\eta)\right)<0,
\end{align}
which asserts that $\mathcal{B}(k,\eta)$ must decrease along the system trajectories when these trajectories are near the boundary.

According to \eqref{eq:BarrierFunction} and \eqref{eq:SystemUpdate1}-\eqref{eq:SystemUpdate2}, the time derivative of the barrier function $\mathcal{B}(k,\eta)$ w.r.t. the auxiliary time variable $\eta$ is
\begin{align}\label{eq:Bdot}
\frac{d}{d\eta}\mathcal{B}(k,\eta)=&-\sigma\left\Vert\nabla_{\hat{v}}\mathcal{B}(k,\eta)\right\Vert^2+\sigma\left\Vert\nabla_{\hat{\lambda}}\mathcal{B}(k,\eta)\right\Vert^2\nonumber\\
&+\sigma\left(\nabla_{\hat{\lambda}}\mathcal{B}(k,\eta)\right)^\top\Psi(k,\eta).
\end{align}

The limiting behavior of $\left(\nabla_{\hat{\lambda}}\mathcal{B}(k,\eta)\right)^\top\Psi(k,\eta)$, $\left\Vert\nabla_{\hat{\lambda}}\mathcal{B}(k,\eta)\right\Vert^2$, and $\left\Vert\nabla_{\hat{v}}\mathcal{B}(k,\eta)\right\Vert^2$ as $\phi_{i_js_j}(k,\eta)\rightarrow0^+,~\forall j$ is characterized as in \eqref{eq:limitbehavior1}, \eqref{eq:limitbehavior2}, and \eqref{eq:limitbehavior3}, respectively\footnote{Given $f:\mathbb{R}^m\rightarrow\mathbb{R}$, $f(x)=O(1)$ means that $\exists M>0$ such that $|f(x)|<M$.}:
\begin{align}
\left(\nabla_{\hat{\lambda}}\mathcal{B}(k,\eta)\right)^\top\Psi(k,\eta)=&O(1),\label{eq:limitbehavior1}\\
\left\Vert\nabla_{\hat{\lambda}}\mathcal{B}(k,\eta)\right\Vert^2=&\sum\limits_{j=1}^\xi\left(\log(\phi_{i_js_j}(k,\eta))\right)^2+O(1)\nonumber\\
\leq&\xi\left(\log(\underline{\phi}(k,\eta))\right)^2+O(1),\label{eq:limitbehavior2}
\end{align}
where $\underline{\phi}(k,\eta):=\min_{j\in\{1,\cdots,\xi\}}\{\phi_{i_js_j}(k,\eta)\}$. It is clear that $\underline{\phi}(k,\eta)\rightarrow0^+$ as $\phi_{i_js_j}(k,\eta)\rightarrow0^+,~\forall j$. Note that \eqref{eq:limitbehavior1} is deduced according to Remark \ref{remark:dualparameter}.

Note that \eqref{equ:condition} needed for the Lemma presented in Appendix is guaranteed by the Alexandrov's theorem \cite[pp. 333]{Alexandrov2005}. Indeed, since the feasible set of the optimization problem \eqref{eq:CG3} is a convex polyhedron, if \eqref{equ:condition} does not hold, then the outward vectors normal to the faces associated with the active constraints at a boundary point are linearly dependent with positive coefficients; this is only possible if the polyhedron has empty interior, while we assume the interior to be nonempty (see Subsection \ref{sec:ConstraintTightening}).

Now, consider \eqref{eq:limitbehavior3}. Since $\hat{\lambda}_{i_js_j}(\eta)\geq\underline{\lambda}$ as $\phi_{i_js_j}(k,\eta)\rightarrow0^+,~\forall j$ (see Remark \ref{remark:dualparameter}) and $\underline{\phi}(k,\eta)/\phi_{i_js_j}(k,\eta)=1$ for at least some $j\in\{1,\cdots,\xi\}$, by applying the Lemma presented in Appendix we obtain that there exists $\zeta>0$ such that the limiting behavior of $\left\Vert\nabla_{\hat{v}}\mathcal{B}(k,\eta)\right\Vert^2$ as $\phi_{i_js_j}(k,\eta)\rightarrow0^+,~\forall j$ satisfies
\setcounter{equation}{34}
\begin{align}
\left\Vert\nabla_{\hat{v}}\mathcal{B}(k,\eta)\right\Vert^2\geq\frac{\zeta\beta^2\underline{\lambda}^2}{\left(\underline{\phi}(k,\eta)\right)^2}+O(1).
\end{align}

Therefore, by taking the limit from both sides of \eqref{eq:Bdot} as $\phi_{i_js_j}(k,\eta)\rightarrow0^+,~j\in\{1,\cdots,\xi\}$, we obtain\footnote{Since 
$\lim_{x\rightarrow0^+}b\left(\log(x)\right)^2/\left(a/x^2\right)=0$ for any $a,b>0$ (one can use the L'H\^{o}pital's rule to show that), the function $\frac{a}{x^2}$ grows faster than $b(\log(x))^2$ as $x\rightarrow0^+$. Thus, $\lim_{x\rightarrow0^+}a/x^2-b\left(\log(x)\right)^2>0$ for any $a,b>0$.} equation \eqref{eq:limitB2}, which affirms \eqref{eq:limitB1}. This completes the proof.  
\end{proof}


\begin{remark}\label{remark:implementation}
For discrete-time implementation of system \eqref{eq:SystemUpdate1}-\eqref{eq:SystemUpdate2}, one can use the difference quotient with a sufficiently small sampling period $\Delta\eta$. In this case, a practical approach to prevent constraint violation due to discretization is to further tighten the constraints of \eqref{eq:CG4} as $\log\big(-\beta f_{is}(z(k),v)+1\big)\geq\vartheta,~\forall i,s$, where $\vartheta>0$ is small. Our numerical experiments suggest that such a discrete-time implementation maintains the desired properties of our algorithm.
\end{remark}

\subsection{Acceptance/Rejection Mechanism}\label{sec:ARM}
Theorem \ref{theorem:convergence} showed that at any sampling instant $k$, $(\hat{v}(\eta),\hat{\lambda}(\eta))\rightarrow(v^\dag(k),\lambda^\dag(k))$ exponentially fast as $\eta\rightarrow\infty$. However, due to limited availability of the computational power \eqref{eq:SystemUpdate1}-\eqref{eq:SystemUpdate2} may not have sufficient time to converge to the optimal solution $(v^\dag(k),\lambda^\dag(k))$ at every sampling instant. Indeed, it is more likely that the evolution of system \eqref{eq:SystemUpdate1}-\eqref{eq:SystemUpdate2} terminates before convergence. Since, in general, the behavior of $\hat{v}(\eta)-v^\dag(k)$ is not monotonic, there is a need for a logic-based method to accept or reject $\hat{v}(\eta)$ once \eqref{eq:SystemUpdate1}-\eqref{eq:SystemUpdate2} is terminated.

In this paper, we adopt the acceptance/rejection mechanism presented in \cite{Emanuele2021}. This mechanism relies on the fact that $v(k-1)$ is a feasible and a sub-optimal solution\footnote{Due to limited computational power, the applied command signal at sampling instant $k$ is not necessarily the optimum. For this reason, we drop the $\dag$ when referring to the applied command signal at sampling instant $k$. We do the same when referring to dual parameter.} for \eqref{eq:CG4} at sampling instant $k$ (see Remark \ref{remark:Slater}). Given the termination time $\eta_t$, the acceptance/rejection mechanism accepts $\hat{v}(\eta_t)$ (i.e., sets $v(k)=\hat{v}(\eta_t)$) if $\hat{v}(\eta_t)$ satisfies the following condition
\setcounter{equation}{36}
\begin{align}\label{eq:ARmechanism}
\left\Vert \hat{v}(\eta_t)-r(k)\right\Vert_Q^2\leq&\left\Vert v(k-1)-r(k)\right\Vert_Q^2-\left\Vert \hat{v}(\eta_t)-v(k-1)\right\Vert_Q^2,
\end{align}
and rejects (i.e., sets $v(k)=v(k-1)$) otherwise. Note that, as shown in \cite{Emanuele2021}, the condition \eqref{eq:ARmechanism} holds for $v^\dag(k)$ and any feasible $v(k-1)$, meaning that the mechanism does not discard the optimal solution if system \eqref{eq:SystemUpdate1}-\eqref{eq:SystemUpdate2} converges.


\subsection{Warm-starting}\label{sec:WarmStarting}
Theorem \ref{theorem:convergence} showed that dynamics evolving according to \eqref{eq:SystemUpdate1}-\eqref{eq:SystemUpdate2} converge exponentially fast. The inequality given in  \eqref{eq:ExponentialConvergence} indicates that faster convergence occurs for larger $\sigma$ and/or $\mu$. If $\sigma$ is made large, high update rate will be necessary when implementing \eqref{eq:SystemUpdate1}-\eqref{eq:SystemUpdate2} in discrete-time. Also, $\mu$ is determined by problem characteristics and is not directly tunable. This underlines the importance of warm-starting to improve convergence.

Regarding $\hat{v}(0)$, note that the set $\tilde{O}_{\infty,1/\beta}$ is positively invariant (see Remark \ref{remark:saddle}). Thus, it is desirable to set $\hat{v}(0)$ to the previously applied command signal (i.e., $\hat{v}(0)=v(k-1)$ at sampling instant $k$). This selection is reasonable, as in most applications, from one sampling instant to the next the state of the system $z(k)$ and the reference signal $r(k)$ do not change substantially.

Regarding $\hat{\lambda}(0)$, any non-negative value, including $\lambda(k-1)$, is feasible. Though $\lambda(k-1)$ is often a good approximation for the optimum dual variables at sampling instant $k$, there is an opportunity for improving the initial guess, as shown below.

As mentioned in Remark \ref{remark:KKT}, $\lambda_{is}(k-1)\neq0$ means that the constraint on the $i$-th output at prediction time $s$ was active at sampling instant $k-1$. This condition moves one step backward at sampling instant $k$; that is the constraint on the $i$-th output at prediction time $s-1$ will be active. The same condition holds for inactive constraints, i.e., those associated with $\lambda_{is}(k-1)=0$. This implies that there is a one-step time shift in the active and inactive constraints. Based upon this insight, we propose the following initial condition for $\hat{\lambda}_i(0)=[\hat{\lambda}_{i0}(0)~\cdots~\hat{\lambda}_{is^\ast}(0)~\hat{\lambda}_{i\infty}(0)]^\top,~i\in\{1,\cdots,m\}$:
\begin{align}\label{eq:InitialConditionLambda}
\hat{\lambda}_{i}(0)=[\hat{\lambda}_{i1}(k-1)~\cdots~\hat{\lambda}_{is^\ast}(k-1)~\hat{\lambda}_{is^\ast}(k-1)~\hat{\lambda}_{i\infty}(k-1)]^\top,
\end{align}
where $\hat{\lambda}_{is^\ast}(k-1)\geq0$ is used as an initial guess for the value of dual parameter at the new prediction time. We have found that in our experiments such a warm-starting worked well.

\subsection{ROTEC}\label{sec:ROTEC}
The ROTEC algorithm is presented in Algorithm \ref{alg:algorithm}. This algorithm should be run at every sampling instant. This algorithm computes the control input $u(k)$ at every sampling time and provides the initial condition $(\hat{v}(0),\hat{\lambda}(0))$ for the next sampling time. Algorithm \ref{alg:algorithm} addresses system safety by ensuring constraint satisfaction at all times (see Theorem \ref{theorem:convergence1}), assuming that the discrete-time implementation accurately approximates the continuous-time updates. It also addresses system performance by optimizing the applied command signal (see Theorem \ref{theorem:convergence}). Finally, Algorithm \ref{alg:algorithm} addresses real-time schedulability, as it is robust to early termination; this allows us to choose the sampling period of tasks to satisfy the schedulability condition \eqref{eq:schedulabilitycondition} with no concern about its performance.

Note that once task $\tau_N$ is terminated, there will be no time to implement the mechanism \eqref{eq:ARmechanism} and compute corresponding control input. To address this issue, during the run-time of the virtual system (step 3), we continuously implement the mechanism (step 4) and update and store in memory the control input (step 5) without applying it to the system. This will guarantee that once task $\tau_N$ is terminated, a suitable control input is available without requiring any further computations.

\begin{algorithm}[!t]
\caption{ROTEC}\label{alg:algorithm}
\begin{algorithmic}[1]
\Require{State of system \eqref{eq:systemfinal} at sampling instant $k$ (i.e., $z(k)$), and the applied command signal and the obtained dual parameter at the previous sampling instant (i.e., $v(k-1)$ and $\lambda(k-1)$)}
\Ensure{The control input at sampling instant $k$ (i.e., $u(k)$)}
\Procedure{ROTEC}{$z(k),v(k-1),\lambda(k-1)$}
\While{Task $\tau_N$ is not terminated}
\State \multiline{Update system \eqref{eq:SystemUpdate1}-\eqref{eq:SystemUpdate2} with a given $\Delta\eta$, and initial condition $\hat{v}(0)=v(k-1)$ and $\hat{\lambda}(0)$ as in \eqref{eq:InitialConditionLambda}.}
\State \multiline{Implement the acceptance/rejection mechanism given in \eqref{eq:ARmechanism} at every virtual time step (i.e., $\alpha\Delta\eta,~\alpha=0,1,\cdots$).}
\State \multiline{Compute and store control input $u(k)$ via \eqref{eq:controllaw} at every virtual time step if the condition \eqref{eq:ARmechanism} is
satisfied.}
\EndWhile
\State \textbf{end while}
\State \textbf{return} command signal $v(k)$ and control input $u(k)$
\EndProcedure
\end{algorithmic}
\end{algorithm}

\section{Simulation Study\textemdash Vehicle Rollover Prevention}\label{sec:simulation}
Rollover is a safety issue for a vehicle \cite{Ataei2020,Shi2021}, in which it tips over onto its side or roof. In this section, we use a simplified model to represent the vehicle dynamics, and apply Algorithm \ref{alg:algorithm} to guard the vehicle against rollover. Note that vehicles are prototypical examples of systems with limited computing power processors where execute multiple parallel functions \cite{Ibrahim2020,Kim2021}.


\subsection{Setting}
We consider a scenario where the longitudinal speed is constant and two safety-critical control systems are implemented on a single processor. The sampling period of the first task is 100 [ms], and its execution time (expressed in ms) follows a Weibull distribution\footnote{Using the Weibull distribution to characterize the execution time of a task is well-accepted in real-time scheduling literature (e.g., \cite{Ophelders2008,Lu2012}).} with shape parameter 2, location parameter 20, and scale parameter 4. Thus, the worst-case execution time of the first task is 30 ms. The second task tracks a desired Steering Wheel Angle (SWA) which is generated by either a human driver or a higher-level controller. We employ ROTEC to manipulate the applied SWA (i.e., the command signal) to prevent rollover, while ensuring convergence to the desired SWA.

As shown in \cite{Bencatel2018}, the vehicle dynamics can be modelled as $\dot{x}=A_ox+B_o\cdot SWA(t)$, where $x(t)=[x_1(t)~x_2(t)~x_3(t)~x_4(t)]^\top$ with $x_1(t)$ as the roll angle, $x_2(t)$ as the roll rate, $x_3(t)$ as the lateral velocity, and $x_4(t)$ as the yaw rate. We consider the one-sample delay described in Subsection \ref{sec:TaskN}. Given that the vehicle has a constant speed of 50 [m/h]\footnote{This assumption is reasonable, as we can assume that the vehicle tracks a constant longitudinal speed through a separate control law.}, $A_o$ and $B_o$ are \cite{Liu2020}:
\begin{align*}
A_o=&\left[\begin{matrix}0.00499 & 0.997 & 0.0154 & -6.81\times10^{-5}\\ -78.3 & -12.2 & -65.3 & -3.89 \\ -0.932 & -0.799 & -6.20 & -1.57 \\ 1.52 & 3.32 & 8.27 & -1.49\end{matrix}\right],\\ B_o=&\left[\begin{matrix}-5.76\times10^{−5} & 2.80 & 0.278 & 0.655\end{matrix}\right]^\top.
\end{align*}

The rollover constraints are defined through the Load Transfer Ratio (LTR), and are imposed as $\left\vert LTR(t)\right\vert\leq1$, where $LTR(t)$ for the given longitudinal speed is $LTR(t)=0.12x_1(t)+0.0124x_2(t)-0.0108x_3(t)+0.0109x_4(t)$.

We use \texttt{YALMIP} toolbox \cite{Lofberg2004} to implement the computations of the conventional CG scheme. The worst-case execution time for the conventional CG from 2000 runs is $\sim$200 ms. Thus, the sampling period of the conventional CG should be $>285$ ms.

\subsection{System Performance\textemdash Comparison Study}
We consider the following three cases: Case I) There is no computational limitation (for instance, we implement the tasks on a more powerful processor), and we implement the conventional CG every 100 milliseconds; Case II) We implement the conventional CG, and to satisfy the real-time schedulability condition \eqref{eq:schedulabilitycondition}, we let the sampling period of the second task to be 300 [ms]; and Case III) We set the sampling period of the second task to 100 [ms], and implement ROTEC with $\sigma=100$, $\beta=10^5$, and $\Delta\eta=0.001$. For comparison purposes, we define the performance index as $\text{PI}\triangleq\int\left\Vert v(t)-r(t)\right\Vert^2dt$, where the integration is performed over the duration of the simulations.

Simulation results are shown in \figurename~\ref{fig:Simulation1}, where 2000 runs are presented for Case III. The normalized achieved PIs for all cases are reported in TABLE~\ref{tab:TableComparison}, where the achieved PI for Case I is used as the basis for normalization. As seen in this table, using a large sampling period (Case II) can degrade the performance. ROTEC (Case III) yields a better performance by computing a sub-optimal solution every 100 ms.

\begin{figure}[!t]
\centering
\includegraphics[width=4.25cm]{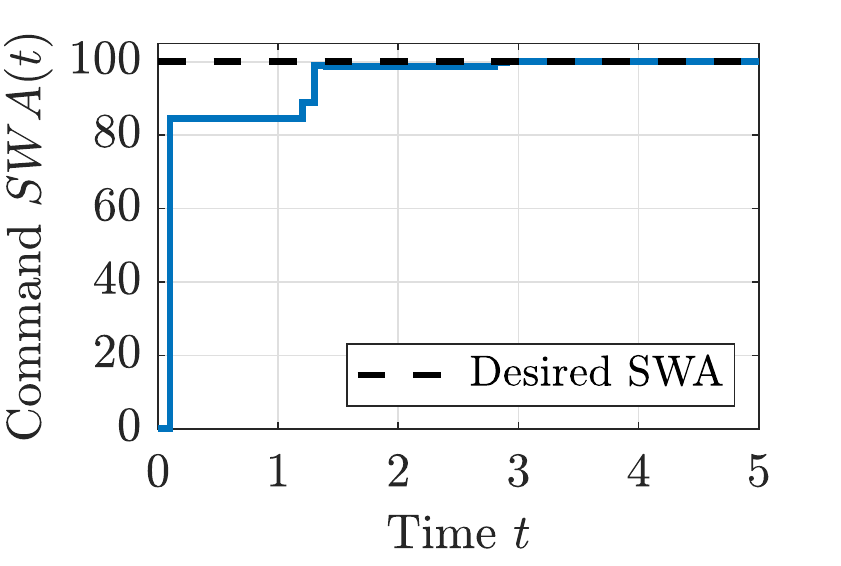}\includegraphics[width=4.25cm]{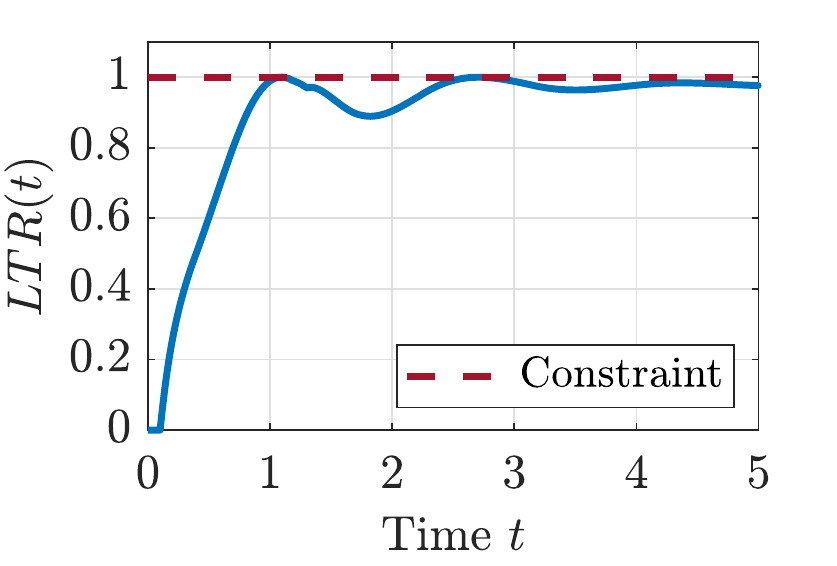}\\(a)\\
\includegraphics[width=4.25cm]{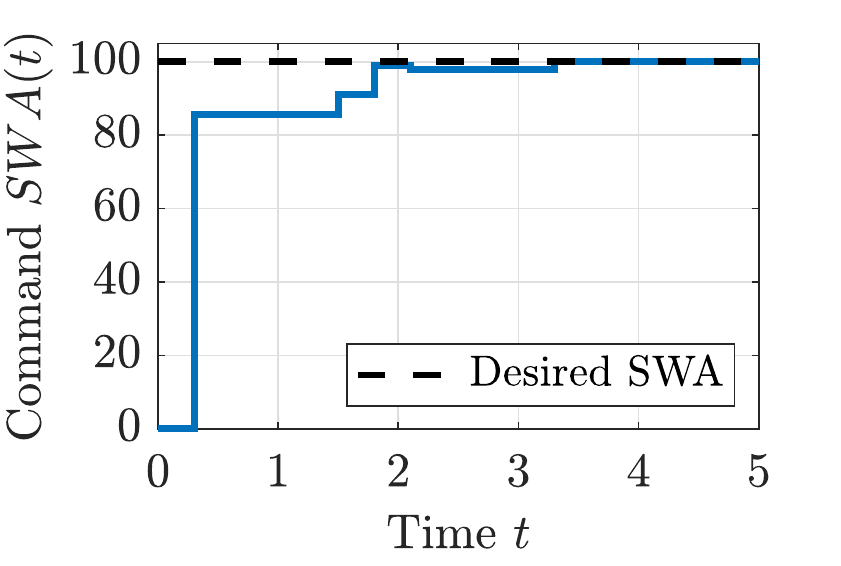}\includegraphics[width=4.25cm]{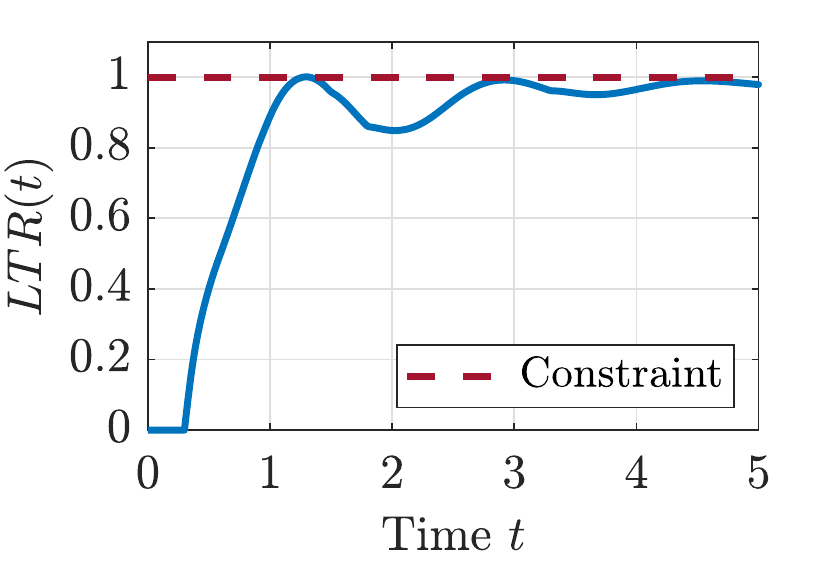}\\(b)\\
\includegraphics[width=4.25cm]{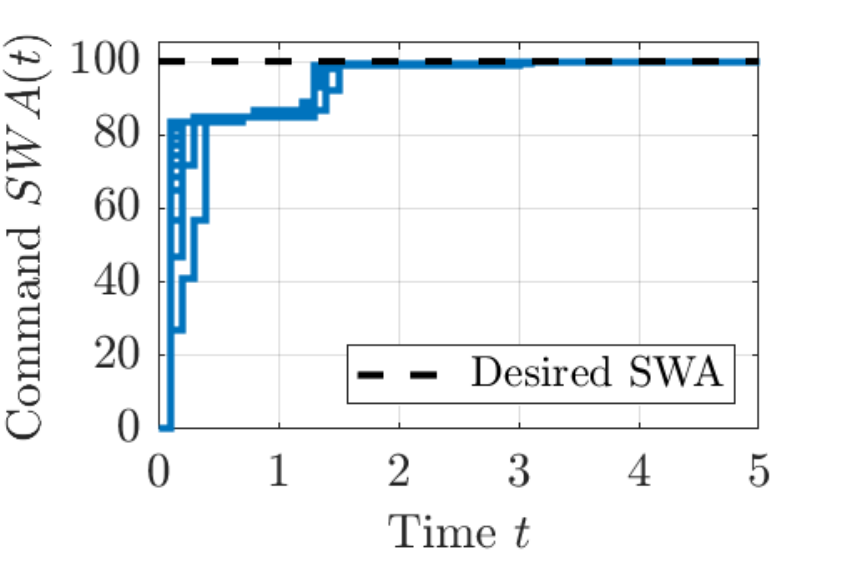}\includegraphics[width=4.25cm]{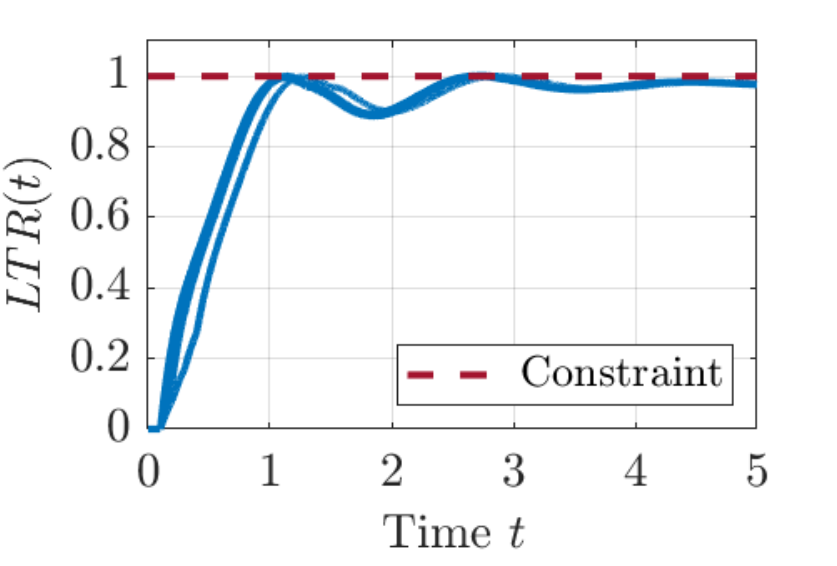}\\(c)
\caption{Figure (a): Results for Case I, in which the conventional CG is implemented every 100 ms. Figure (b): Results for Case II, in which the conventional CG is implemented every 300 ms. Figure (c): Results for Case III (2000 runs), in which ROTEC is implemented every 100 ms.}
\label{fig:Simulation1}
\end{figure}

\begin{table}[!t]
\centering
\caption{System Performance Analysis for Case I, II, and III}
\label{tab:TableComparison}
\begin{tabular}{c|c|c|c}
     & Case I & Case II & Case III \\
     \hline
Normalized PI & 1 & 1.82 & 1.34 (Mean)
\end{tabular}
\end{table}

\subsection{Rollover Prevention\textemdash Fishhook Test}
To investigate the constraint-handling property of ROTEC, we conduct the fishhook test \cite{Dahmani2016,Mehrtash2019}. This test is a steer/countersteer maneuver, in which to achieve the maximum severity, the desired SWA switches when the maximum roll angle is reached. This test is described in \figurename~\ref{fig:Fishhook}.

Simulation results are shown in \figurename~\ref{fig:Simulation2}. As seen in this figure, ROTEC ensures that the vehicle will not roll over when subject to a severe obstacle avoidance maneuver.

\subsection{Impact of A High Number of Early Terminations}
In this subsection, we study the impact of a high number of early terminations on the performance of ROTEC. To conduct this study, we discretize the vehicle dynamics with a sampling period of 100 msec, but we limit the execution time of ROTEC to 10 $\mu$sec. Such a limitation ensures that the CG task faces early termination at most of the sampling instants. A high number of early terminations prevents ROTEC from converging to the optimal solution at most of the sampling instants. As seen from \figurename~\ref{fig:Comparison1}, this degrades the tracking performance and slows down the convergence of the CG output to the reference command.  Nevertheless, the constraints are not violated and convergence of CG output to the reference command is still achieved.

\subsection{Sensitivity Analysis\textemdash Impact of $\sigma$}
We conducted sensitivity analysis of the performance of ROTEC with respect to the design parameter $\sigma$. \figurename~\ref{fig:Simulation3} shows how $\sigma$ impacts the performance. From \figurename~\ref{fig:Simulation3} we see that as $\sigma$ decreases, the performance of ROTEC degrades. This is consistent with expectations from \eqref{eq:ExponentialConvergence}. Note that a large $\sigma$ also reduces the number of discarded $\hat{v}(\eta)$ as a result of violation of \eqref{eq:ARmechanism}, such that the mean number of rejected command signals is $4$ and $0$ for $\sigma=50$ and $\sigma=150$, respectively.

\begin{figure}[!t]
\centering\includegraphics[width=\columnwidth]{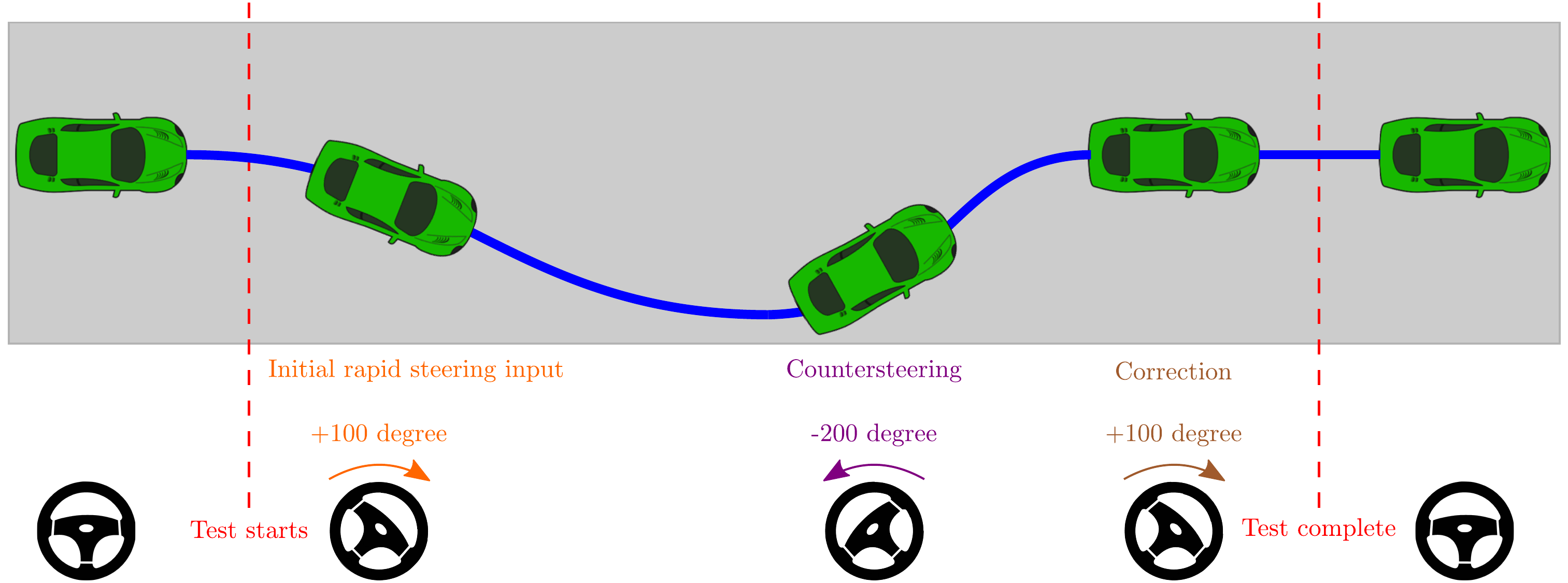}
\caption{The considered fishhook test for investigating rollover prevention.}
\label{fig:Fishhook}
\end{figure}

\begin{figure}[!t]
\centering\includegraphics[width=4.25cm]{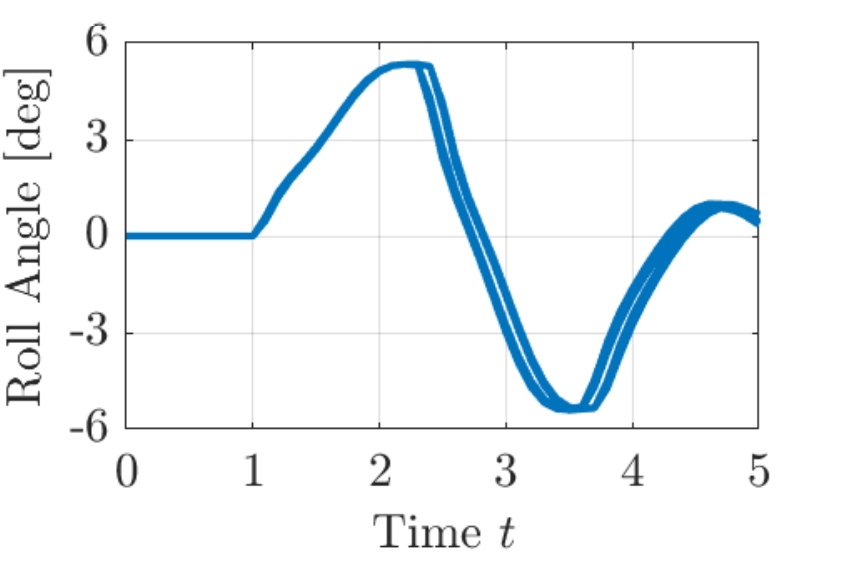}\includegraphics[width=4.25cm]{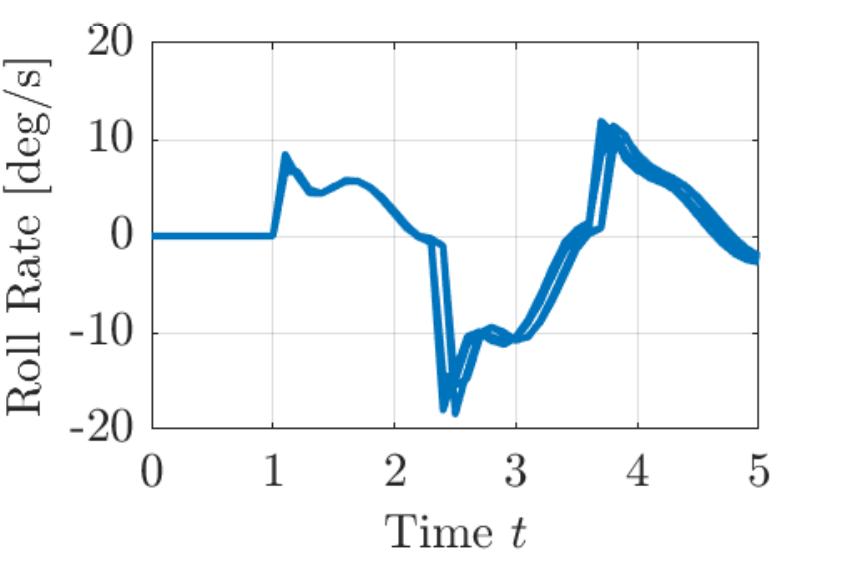}\\
\includegraphics[width=4.25cm]{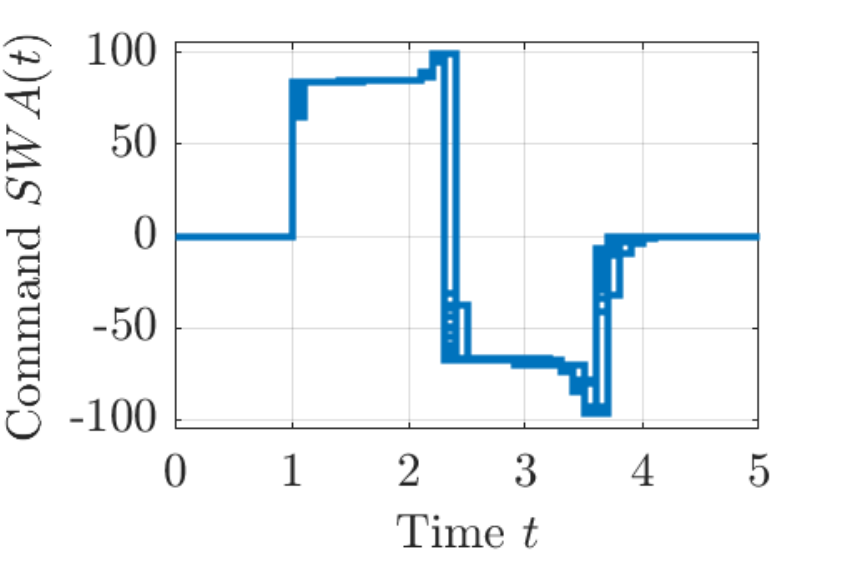}\includegraphics[width=4.25cm]{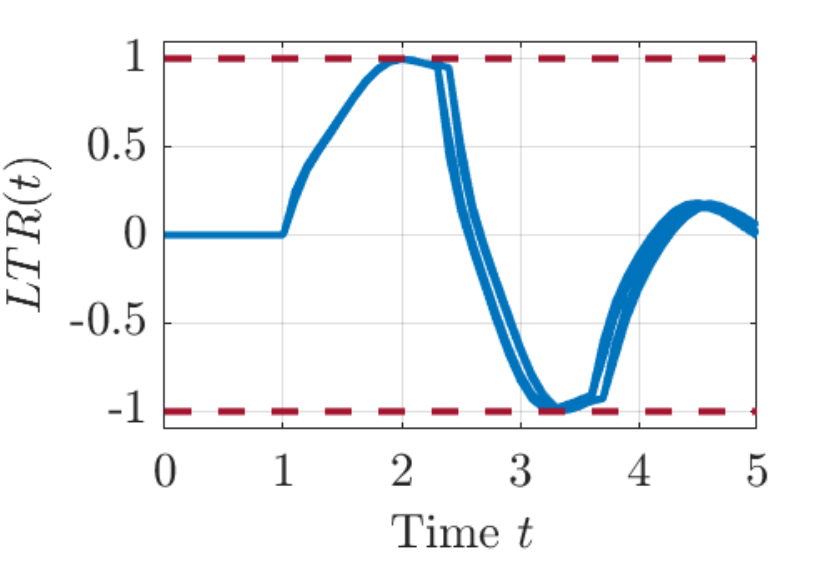}
\caption{Simulation results for the fishhook test (2000 runs). The desired SWA switches when the roll angle reaches its maximum (i.e., $x_2(t)=0$).}
\label{fig:Simulation2}
\end{figure}

\begin{figure}[!t]
\centering
\includegraphics[width=8cm]{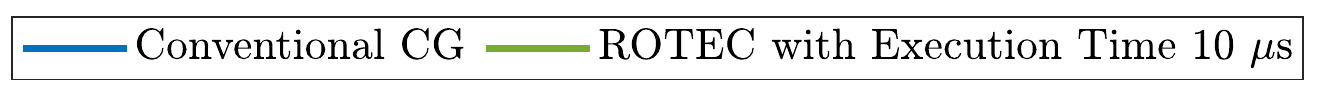}\\
\includegraphics[width=4.25cm]{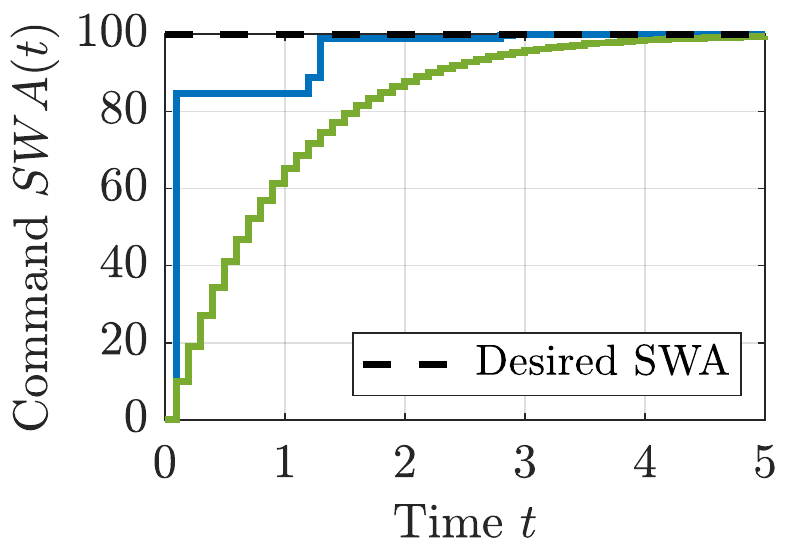}\hspace{0.2cm}\includegraphics[width=4.25cm]{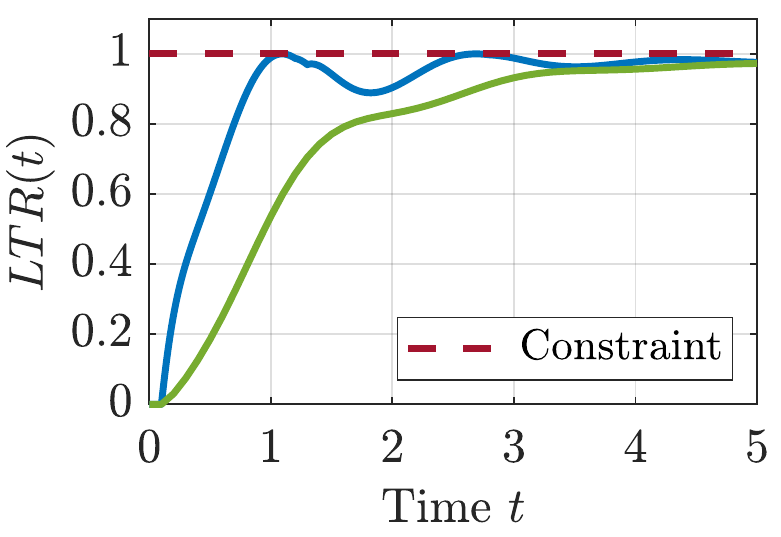}
\caption{The impact of a high number of early terminations on the tracking performance obtained by ROTEC.}
\label{fig:Comparison1}
\end{figure}

\begin{figure}[!t]
\centering
\includegraphics[width=8cm]{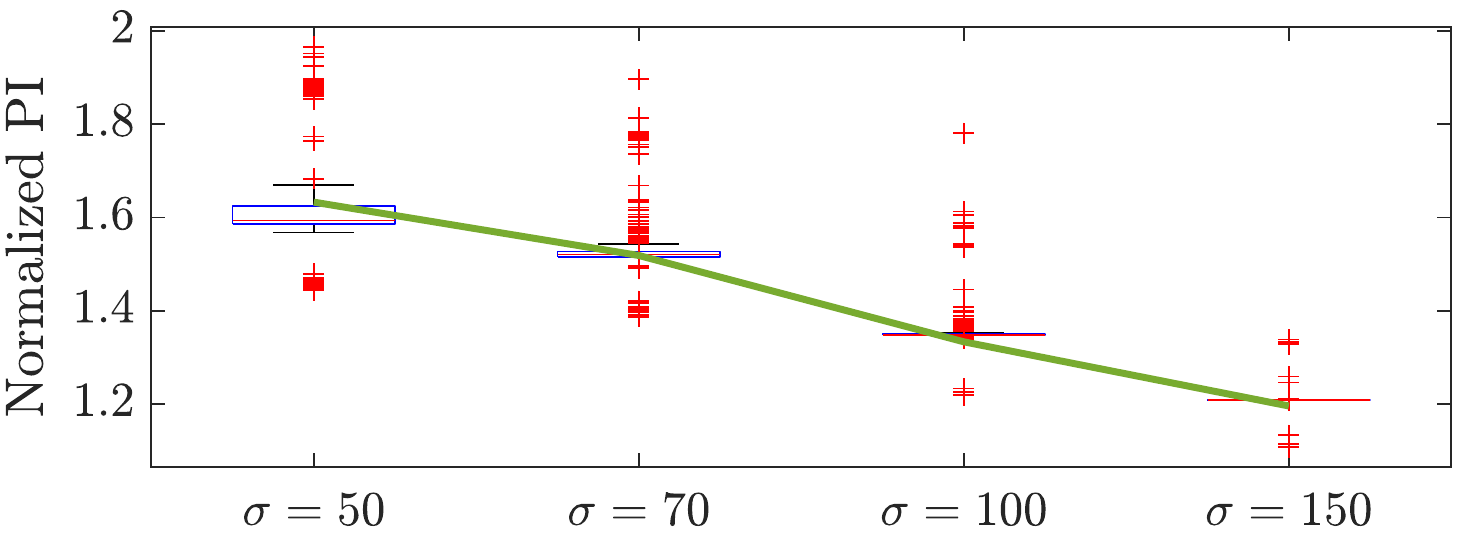}
\caption{The impact of $\sigma$ on the obtained PI from 2000 runs. The green line shows the mean values.}
\label{fig:Simulation3}
\end{figure}

\subsection{Sensitivity Analysis\textemdash Impact of Sampling Period}
In this subsection, we study the sensitivity of ROTEC to the sampling period. For this study, we assume that the vehicle dynamics are discretized with different sampling periods, and ROTEC is the only control system running on the processor. \figurename~\ref{fig:Simulation4} shows how sampling period impacts the performance of closed-loop system. As discussed in Section \ref{sec:PF}, the larger the sampling period is, the poorer performance of the CG scheme could become. This is illustrated in \figurename~\ref{fig:Simulation4}.

\begin{figure}[!t]
\centering
\includegraphics[width=8cm]{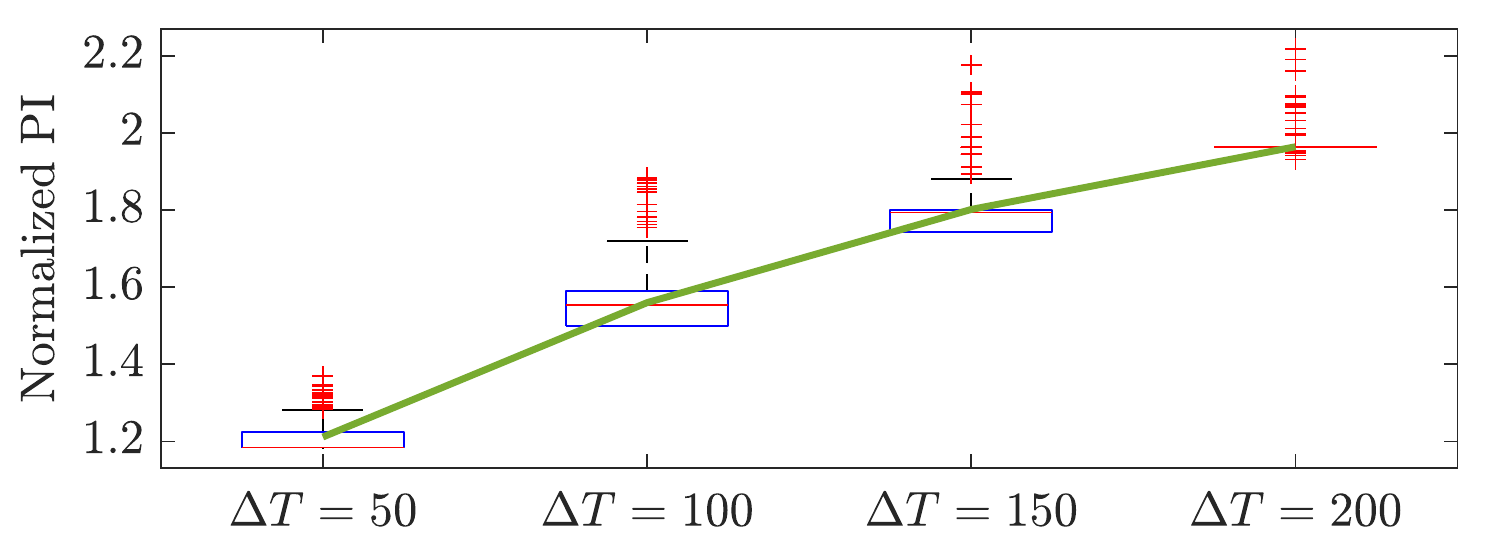}
\caption{The impact of sampling period (expressed in ms) on the obtained PI from 2000 runs. The green line shows the mean values.}
\label{fig:Simulation4}
\end{figure}

\section{Conclusion}\label{sec:conclusion}
This paper proposed ROTEC (RObust to early TErmination Command governor), an algorithm capable of maintaining feasibility of command governor by adapting to available computation times. Variability in time available to
perform computations is a common occurrence in modern Cyber-Physical Systems, where several tasks can run on the same processor. The core idea of ROTEC is to use a continuous-time primal-dual gradient flow algorithm that is run for as long as the processor is available for computation, and to augment such an algorithm with an acceptance/rejection logic. ROTEC guarantees constraint satisfaction at all times, and provides a sub-optimal but feasible and effective solution if early terminated due to limited computation time. The effectiveness of ROTEC is validated through simulation studies of vehicle rollover prevention. The paper shows that ROTEC addresses the availability of limited computating power, yields an acceptable performance, and guarantees rollover prevention under severe steer/countersteer maneuvers.

Future research will consider how other optimization-based controllers can be implemented in a way that is robust to early termination and variability in available processor power.

\bibliographystyle{IEEEtran}
\bibliography{ref}{}

\section*{Appendix}\label{appendix}
\begin{lemma-non} 
Let $M_i\in\mathbb{R}^{p}$, and $m_i\geq0$, $i=1,\cdots,q$ be given vectors and real numbers, respectively. Suppose that it is known that $m_{i^\ast}\geq\underline{m}>0$ for some $i^*\in\{1,\cdots,q\}$. Let 
$$\Omega_i=\left\{W\in\mathbb{R}^{p}\Bigg| W=\sum\limits_{\substack{j=1\\j\neq i}}^{q}m_jM_j~\text{for some~}m_j\geq0\right\},
$$ 
and assume that
\begin{equation}\label{equ:condition}\tag{Condition $\maltese$}
-M_{i}\not\in\Omega_{i},~i\in\{1,\cdots,q\},
\end{equation}

Let
$$
J=\left(\sum_{i=1}^{q}m_iM_i\right)^\top \left(\sum_{i=1}^{q}m_iM_i\right).
$$

Then, there exists
$\epsilon>0$  such that for all $m_i$ satisfying the above assumptions the value
of $J$ admits the following bound:
$$
J\geq\epsilon\underline{m}^2>0.
$$
\end{lemma-non}

\begin{proof}
Since $m_{i^\ast}\geq\underline{m}>0$, one can rewrite $J$ as
\begin{align}
J=&m_{i^\ast}^2\left(M_{i^\ast}+\sum\limits_{\substack{i=1\\i\neq i^\ast}}^{q}\frac{m_i}{m_{i^\ast}}M_i\right)^\top\left(M_{i^\ast}+\sum\limits_{\substack{i=1\\i\neq i^\ast}}^{q}\frac{m_i}{m_{i^\ast}}M_i\right)\nonumber\\
\geq&\underline{m}^2\left\Vert-M_{i^\ast}-\sum\limits_{\substack{i=1\\i\neq i^\ast}}^{q}\frac{m_i}{m_{i^\ast}}M_i \right\Vert^2.\nonumber
\end{align}

Since $-M_{i^\ast}\not\in\Omega_{i^\ast}$ and $\Omega_{i^\ast}$ is a closed set, there is a minimum distance $d_{i^\ast}>0$ between $-M_{i^\ast}$ and the set $\Omega_{i^\ast}$. Thus, letting $\epsilon=d_{i^\ast}^2$ completes the proof. 
\end{proof}

\end{document}